\documentclass[11pt,reqno]{amsart}
\usepackage{lipsum}
\usepackage{amsaddr}
\usepackage{amsfonts}
\usepackage{amssymb}
\usepackage{amsmath}
\usepackage{a4wide}
\usepackage{pgf,pgfarrows,pgfautomata,pgfheaps,pgfnodes,pgfshade}
\usepackage{url}
\usepackage{lineno}
\usepackage{scalerel}
\usepackage{stackengine,wasysym}
\usepackage{verbatim}

\usepackage{float}

\newcommand\reallywidetilde[1]{\ThisStyle{%
  \setbox0=\hbox{$\SavedStyle#1$}%
  \stackengine{-.1\LMpt}{$\SavedStyle#1$}{%
    \stretchto{\scaleto{\SavedStyle\mkern.2mu\AC}{.5150\wd0}}{.6\ht0}%
  }{O}{c}{F}{T}{S}%
}}

\usepackage[numbers]{natbib}

\newcommand*\patchAmsMathEnvironmentForLineno[1]{%
  \expandafter\let\csname old#1\expandafter\endcsname\csname #1\endcsname
  \expandafter\let\csname oldend#1\expandafter\endcsname\csname end#1\endcsname
  \renewenvironment{#1}%
     {\linenomath\csname old#1\endcsname}%
     {\csname oldend#1\endcsname\endlinenomath}}%
\newcommand*\patchBothAmsMathEnvironmentsForLineno[1]{%
  \patchAmsMathEnvironmentForLineno{#1}%
  \patchAmsMathEnvironmentForLineno{#1*}}%
\AtBeginDocument{%
\patchBothAmsMathEnvironmentsForLineno{equation}%
\patchBothAmsMathEnvironmentsForLineno{align}%
\patchBothAmsMathEnvironmentsForLineno{flalign}%
\patchBothAmsMathEnvironmentsForLineno{alignat}%
\patchBothAmsMathEnvironmentsForLineno{gather}%
\patchBothAmsMathEnvironmentsForLineno{multline}%
}

\usepackage{hyperref}
\numberwithin{equation}{section}

\newtheorem{theorem}{Theorem}[section]
\newtheorem{lemma}[theorem]{Lemma}

\theoremstyle{definition}
\newtheorem{definition}[theorem]{Definition}
\newtheorem{example}[theorem]{Example}
\newtheorem{corollary}[theorem]{Corollary}
\newtheorem{assumption}[theorem]{Assumption}
\newtheorem{proposition}[theorem]{Proposition}

\theoremstyle{remark}
\newtheorem{remark}[theorem]{Remark}

\numberwithin{equation}{section}

\newcommand{\thmref}[1]{Theorem~\ref{#1}}
\newcommand{\lemref}[1]{Lemma~\ref{#1}}
\newcommand{\corref}[1]{Corollary~\ref{#1}}

\newcommand{\propref}[1]{Proposition~\ref{#1}}
\newcommand{\defref}[1]{Definition~\ref{#1}}
\newcommand{\eqnref}[1]{(\ref{#1})}
\newcommand{\assref}[1]{Assumption~\ref{#1}}

\newcommand{\secref}[1]{Section~\ref{#1}}

\newcommand{\tmL}{\mathcal{L}} 
\newcommand{\mT}{\mathcal{T}} 
\newcommand{\id}{\mathcal{I}} 
\newcommand{\ind}{{\bf 1}} 

\newcommand{\cz}{\mathcal{C}_0 } 
\newcommand{\ninf}[1]{\|#1\|_{\infty}} 
\newcommand{\cc}{\mathcal{C}_c}	
\newcommand{\C}{\mathcal{C}}	
\newcommand{\R}{\mathbb{R}} 
\newcommand{\sE}{\mathsf{E}} 
\newcommand{\sO}{\sE_\partial}
\newcommand{\diff}{\mathrm{d}}

\newcommand{\tmP}{\mathcal{P}} 
\newcommand{\tmR}{\mathcal{R}} 
\newcommand{\mF}{\mathcal{F}}	
\newcommand{\tmG}{\mathcal{G}}
\newcommand{\mE}{\mathcal{E}}

\newcommand{\mN}{\mathbb{N}^+}

\newcommand{\mO}{\mathcal{O}}
\newcommand{\mA}{\mathcal{A}}

\newcommand{\ti}[1]{\tilde{#1}} 
\newcommand{\mP}{\tilde{\mathcal{P}}} 
\newcommand{\mL}{\tilde{\mathcal{L}}} 
\newcommand{\mG}{\tilde{\mathcal{G}}}
\newcommand{\Rom}[1]{\uppercase\expandafter{\romannumeral#1}}

\newcommand{\resp}[1]{\textit{respectively}, #1}

\def\bfE{\mbox{\boldmath$E$}}
\def\bfP{\mbox{\boldmath$P$}}

\begin{document}

\title[Optimal Stopping Problems for Feller Processes]
{Viscosity Solution for Optimal Stopping Problems of Feller Processes}

\author{Suhang Dai}
\address{ Institute for Financial and Actuarial Mathematics, Department of Mathematical Sciences, University of Liverpool, L69 7ZL, United Kingdom}
\email{sgsdai@liverpool.ac.uk}

\author{Olivier Menoukeu-Pamen }
\address{ African Institute for Mathematical Sciences, Ghana\\University of Ghana, Ghana \\ Institute for Financial and Actuarial Mathematics, Department of Mathematical Sciences, University of Liverpool, L69 7ZL, United Kingdom}
\email{menoukeu@liv.ac.uk}
\thanks{The project on which this publication is based has been carried out with funding provided by the Alexander von Humboldt Foundation, under the programme financed by the German Federal Ministry of Education and Research entitled German Research Chair No 01DG15010.\\The authors would like to thank Bernt \O ksendal for his comments and for suggesting problem in Section \ref{refsecappli22}}


\maketitle

\begin{abstract}
We study an optimal stopping problem when the state process is governed by a general Feller process. In particular, we examine viscosity properties of the associated value function with no a priori assumption on the stochastic differential equation satisfied by the state process. Our approach relies on properties of the Feller semigroup. We present conditions on the state process under which the value function is the unique viscosity solution to an Hamilton-Jacobi-Bellman (HJB) equation associated with a particular operator. More specifically, assuming that the state process is a Feller process, we prove uniqueness of the viscosity solution which was conjectured in \cite{palczewski2014infinite}. We then apply our results to study viscosity property of optimal stopping problems for some particular Feller processes, namely diffusion processes with piecewise coefficients and semi-Markov processes. Finally, we obtain explicit value functions for optimal stopping of straddle options, when the state process is a reflected Brownian motion, Brownian motion with jump at boundary and regime switching Feller diffusion, respectively (see Section \ref{secexpsolu}).
\end{abstract}


\section{Introduction}

Optimal stopping problems for Markov processes have been extensively studied in the literature using various methods; see for example \cite{peskirshiryaev06}. Such problems are very important due to their various applications in engineering, physics, mathematical finance and insurance. Assuming that the state process is given by a diffusion process (with non degenerate diffusion coefficient), the pioneering book \cite[Chapter 3]{bensoussan1978applications} introduces a variational inequality approach to solve optimal stopping problems. Under some weak regularity of the data the authors prove the regularity of the value function. Since then, there have been many studies on optimal stopping problems for Markov processes using the variational inequality approach, with the aim of relaxing the assumptions on the class of Markov processes and/or on the reward functional and also studying the properties of the value function. The variational inequality associated to the optimal stopping problem is often difficult to solve, unless one allows a notion of  weak solution, called viscosity solution, to the Hamilton-Jacobi-Bellman (HJB) equation. In the case of a diffusion process, this approach is used for example  in \cite{bassan2002optimal,bassan2002regularity,fleming2006controlled}; see also \cite{oksendal2005applied} for the jump-diffusion case.

In studying the viscosity properties of the value function, the traditional approach assumes that the generator associated with the state Markov process is given by parabolic or elliptic differential operators. Hence, one can use tools from partial differential equations to solve the problem. A natural question is what happens when the state process is given by a Markov process (for example a Feller process) for which the generator is not given by a partial differential operator but only derived from its semigroup. To the best of our knowledge, only \cite{palczewski2014infinite} deals with existence of viscosity solution of an HJB equation when the generator is derived from a Feller semigroup.

One of the main motivation of this paper is to provide a general analytical approach that extends earlier results on properties of the value function to a more general class of processes. As such, we do not assume that the generator of the process is given by a partial differential operator. The other motivation is to establish a framework that enables to find the value function of an optimal stopping problem for a general class of processes (Feller processes) by analytically deriving the unique viscosity solution to the associated HJB equation (compare with \cite{palczewski2014infinite}). Thus, our result completes the previous studies, in the sense that, we derive the existence and uniqueness of viscosity solutions to the HJB equation. The uniqueness was conjectured in \cite{palczewski2014infinite}. To our knowledge, we do not know of any existing results on uniqueness of viscosity solution in this framework.

In this paper, we consider an infinite time horizon optimal stopping problems with fixed discount rate. We use the penalty method introduced in \cite{robin1978controle} and the general setting in \cite{Stettner1980}. Contrary to the traditional method which is based on  calculations of the (integro) differential operators, this method is based on an efficient approximation of the value function by smooth functions. Although there are several extensions of the penalty method (see for example \cite{ palczewski2010finite, palczewski2011stopping, palczewski2014infinite, janluk2010, stettner2011penalty}), most of them focus on the study of the continuity of the value function except work \cite{palczewski2014infinite} which investigates the existence of viscosity solution to the associated HJB equation. In this paper, under slightly different conditions, we show that the value function is the unique viscosity solution to the HJB equation associated with the optimal stopping problem. 

We apply our result to study viscosity properties of the value functions for optimal stopping problems of L\'evy processes, reflected Brownian motion, sticky Brownian motion, diffusion with piecewise coefficients and semi-Markov processes. We show that depending on the choice of the operator and its domain, the value function is the unique viscosity solution associated with the HJB equation. Let us mention that our viscosity analysis on diffusion with piecewise coefficients and semi-Markov processes are typically not investigated in the current literature on optimal stopping problems. In the former case, we will see later (confer \corref{cor:skew1} and \corref{cor:skew2}) that the value function is a viscosity solution associated with a particular operator to an HJB equation. In the latter case, we first use perturbation theory (confer \cite{bottcher2013levy}) to transform the one-dimensional semi-Markov process to a two-dimensional Markov process. Then, we show that the value function of the problem is the unique viscosity solution to the associated HJB equation. Similar optimal stopping problem was studied in \cite{boshuizen1993general,muciek2002optimal} using iterative approach. We also use our results to explicitly derive the value function and the optimal stopping time in the case of a straddle option for the subsequent state processes: reflected Brownian motion (see \corref{cor:exvfrefc1c2}); Brownian motion with jump at boundary (see Proposition \ref{explisoltBMjb}) and regime switching Feller diffusion  (see \corref{corregFdi}). 

The rest of this paper is organized as follows. \secref{sec:p0} introduces terminologies used throughout this paper and then formulate the optimal stopping problem. In \secref{sec:p1}, we study the value function as a viscosity solution to an HJB equation. \secref{uninoncomp} investigates uniqueness of the viscosity solution  and its link to the value function under the assumption that the state space is compact. The proof relies on the comparison theorem (\thmref{thm:comparisonprinciple}). \secref{sec:notc1} examines the extension of the uniqueness to the case of non compact state space.  \secref{structsec} studies the structure of the viscosity solution and its link to the martingale approach. 
In \secref{sec:levyprocess}, we apply our results to study viscosity properties of value functions of optimal stopping problems for some processes satisfying our key assumptions. \secref{secexpsolu} is devoted to the derivation of explicit value function for optimal stopping of a straddle option.

%


\section{Preliminaries and problem formulation}\label{sec:p0}

In this section, we first present some basic definitions and properties of Feller processes and Feller semigroups. Then, we formulate the optimal stopping problems and introduce our main assumptions. For more information on Feller processes, the reader may consult for example \cite[Chapter 17]{kallenberg2006foundations} or  \cite[Chapter 1]{bottcher2013levy}.

\subsection{Preliminaries}
Throughout this paper, we suppose that $\sE$ is a locally compact,
separable metric space with metric $\rho$.  $\mE$ is the $\sigma$-algebra of  the Borel sets of $\sE$. If $\sE$ is not compact, we define $\sO:=\sE\cup\{\partial\}$ as the one point (Alexandorff) compactification of $\sE$, where $\{\partial\}$ is the point at infinity; otherwise, $\{\partial\}$ is an isolated point from $\sE$. In both cases, $\sO$ is compact and metrizable and $\mE_\partial$ denotes the $\sigma$-algebra in  $\sE_\partial$ generated by $\mE$.
We will use the following notations:
\begin{itemize}
\item $B(\sE)$ is the space of all bounded Borel measurable functions on $\sE$;
\item $\C(\sE)$ is the space of all continuous functions on $\sE$;
\item $\cc(\sE):=  \{w\in\C(\sE);  w \mbox{ has compact support}\}$;
\item $\cz(\sE):=  \{w\in\C(\sE);  w \mbox{ vanishes at infinity}\}$;
\item $\C_*(\sE):=  \{w\in\C(\sE);  w \mbox{ converges at infinity}\}$;
\item $\C_b(\sE):= \C(\sE)\cap B(\sE)$;
\item $USC(\sE)$ (\resp{$LSC(\sE)$}) denotes the space Borel-measurable upper (\resp{lower}) semicontinuous function on $\sE$.
\end{itemize}
\begin{remark}
The above definitions imply that $\cc(\sE)\subseteq \cz(\sE)\subseteq \C_*(\sE)\subseteq \C_b(\sE)$. Moreover, if $\sE$ is compact, these spaces coincide.
\end{remark}

Let $\ninf{\cdot}$ be the supremum norm that is for any $w\in B(\sE)$,
$$\ninf{f} := \sup_{x\in \sE}|f(x)|.$$
Equipped with the above norm, $\left(\cz(\sE),\ninf{\cdot}\right)$, $\left(\C_*(\sE),\ninf{\cdot}\right)$ and $\left(\C_b(\sE),\ninf{\cdot}\right)$  are Banach spaces. The relation $``\leq"$ is a partial order on the space of real valued functions on $\sE$ and we have $f\leq g$ if and only if $f(x)\leq g(x)$ for all $x\in \sE$. We now give a series of definitions. 
\begin{definition}(\textsl{Feller Semigroup)}\label{def:c0semigroup}
A collection of bounded linear operators $\{\tmP_t\}_{t\geq 0}$ is called \textsl{Feller semigroups on $\cz(\sE)$}, if it satisfies the following four properties:
\begin{itemize}
\item $\tmP_{t+s} = \tmP_{t}\circ\tmP_{s}$, for all $t,s\geq 0$; $\tmP_0 = \id$, where $\id$ is the identity operator.
\item For each $t\geq 0$, if $w\in \cz(\sE)$, $0\leq w \leq 1$, then, $0\leq \tmP_t w\leq 1$.
\item (\textsl{Feller Property}) $\tmP_t:\cz(\sE)\rightarrow\cz(\sE)$ for all $t\geq 0$.
\item (\textsl{Strong Continuous Property}) $\lim_{t\rightarrow 0^+}\ninf{\tmP_{t}w-w}=0$ for $w\in\cz(\sE)$.
\end{itemize}
Furthermore, a semigroup $\{\tmP_t\}_{t\geq 0}$ is \textsl{conservative} if $\tmP_t 1 = 1$ for all $t\geq 0$.

\end{definition}
\begin{definition}\textup{(Feller Process)}\label{def:fellp}
A \textsl{Feller process} $\{X(t)\}_{t\geq 0}$ is a Markov process whose transition semigroup defined by
$$\tmP_t w(x) := \bfE^x\left[w(X(t))\right]\,\, \mbox{ for any } x\in \sE\mbox{ and }w \in B(\sE)$$
 is a Feller semigroup.
\end{definition}

Based on \defref{def:fellp}, the transition semigroup of a Feller process is conservative.
\begin{definition}\textup{(Infinitesimal Generator)}\label{def:infgendef}
An \textsl{infinitesimal generator} of a Feller semigroup $\{\tmP_t\}_{t\geq 0}$ or a Feller process $\{X(t)\}_{t\geq 0}$ is a linear operator $(\tmL, D(\tmL))$, with $\tmL:D(\tmL)\subseteq \cz(\sE)\rightarrow \cz(\sE)$ defined by
\begin{equation}\label{eqn:infgendef}
\tmL w := \lim_{t\rightarrow 0^+}\frac{\tmP_t w-w}{t}\, \mbox{ for } w\in D(\tmL) ,
\end{equation}
where the domain
$$D(\tmL) := \left\{ w\in\cz(\sE); \mbox{such that the limit in \eqnref{eqn:infgendef} exists in } \cz(\sE) \right\}.$$
\end{definition}
\begin{definition}\textup{(Resolvent)}
A \textsl{resolvent} $\{\tmR_{\lambda}\}_{\lambda>0}$ is defined by
$$\tmR_\lambda w(x):=\int_0^\infty e^{-\lambda t}\tmP_t w(x) \,\diff t \,\mbox{ for }  x\in \sE\mbox{ and }w\in \cz(\sE).$$
\end{definition}

The following resolvent identity equation is satisfied: for any $\lambda,\mu>0$ and $w\in\cz(\sE)$
\begin{equation}\label{eqn:resolventidentity}
\tmR_\lambda w-\tmR_\mu w = (\mu-\lambda)\tmR_\lambda\tmR_\mu w.
\end{equation}

\begin{definition} \textup{(}\textsl{Postive Maximum Principle}\textup{)}
	An operator $(\tmL, D(\tmL))$ satisfies \textsl{positive maximum principle} if $\tmL w(x_0)\leq 0$ for any $w\in D(\tmL )$ with $w(x_0) = \sup_{x\in\sE}w(x)\geq 0$.
\end{definition}

We now state the Hille-Yosida-Ray theorem for strongly continuous semigroup. This theorem gives the relationships among Feller semigroup, generator and resolvent (see \cite[Theorem 1.30]{ bottcher2013levy}) and will play a key role in proving the uniqueness of the viscosity solution.

\begin{theorem} \label{thm:hyr}
Let $(\tmG,D(\tmG))$ be a linear operator on $\cz(\sE)$. $(\tmG,D(\tmG))$ is closable and its closure $(\tmG,D(\tmG))$ is the infinitesimal generator of a Feller semigroup if and only if:
\begin{enumerate}
\item $D(\tmG)$ is dense in $\cz (\sE)$.
\item The range of $\lambda-\tmG$ is dense in $\cz (\sE)$ for all $\lambda>0$.
\item $(\tmG,D(\tmG))$ satisfies the positive maximum principle.
\end{enumerate}
\end{theorem}

The following corollary is from the Hille-Yosida theorem (see for example \cite[ Proposition~4.9 and Theorem~4.10]{taira2004semigroups} ).
\begin{corollary}\label{cor:hy}
Let $(\tmL,D(\tmL))$ be the infinitesimal generator of some Feller semigroup. Then,
\begin{enumerate}
\item $(\tmL,D(\tmL))$ is closed.
\item For each $\lambda>0$, the operator $(\lambda-\tmL)$ is a bijection of $D(\tmL) $ onto $\cz(\sE)$ and its inverse is the resolvent $\tmR_\lambda$, that is for all $w\in\cz(\sE)$ and $v\in D(\tmL)$, we have
\begin{equation}\label{eqn:resolventproperties2}
(\lambda-\tmL)\tmR_{\lambda}w = w\, \mbox{ and }\, \tmR_{\lambda}(\lambda-\tmL)v = v.
\end{equation}
\item For each $\lambda>0$, we have the inequality
\begin{align}\label{eqn:resolventproperties233}
\ninf{\tmR_\lambda}:=\sup_{w\in \cz(\sE)}\frac{\ninf{\tmG_\lambda w}}{\ninf{w}}\leq \frac{1}{\lambda}.
\end{align}
\end{enumerate}
\end{corollary}

Subsequently, we give the definition of the core, which enables to uniquely characterize a Feller semigroup.

\begin{definition}(\textsl{Core})
$(\tmG,D(\tmG))$ is called a \textsl{core} of an infinitesimal generator $(\tmL,D(\tmL))$ if it is a linear closable operator which satisfies $D(\tmG)\subseteq D(\tmL)$ is dense in $\cz(\sE)$ and the closure of $(\tmG,D(\tmG))$ is $(\tmL,D(\tmL))$, that is for any $w\in D(\tmL)$, there exists a sequence $\{w_n\}_{n \in \mN}$ in $D(\tmG)$ such that
$$\lim_{n\rightarrow \infty}(\ninf{w_n-w}+\ninf{\tmG w_n-\tmL w})=0.$$
\end{definition}
By (1) in \corref{cor:hy}, it follows that the infinitesimal generator of a Feller semigroup is its the core.

\subsection{Problem Formulation}

In this paper, we study an optimal stopping problem for a normal Markov process $X:=(\Omega,\mF,\mF_t,X_t,\theta_t,\bfP^x)$ on the state space $(E,\mE)$, where $(\Omega,\mF)$ is a measurable space, $\{\mF_t\}_{t\geq 0}$ is a right continuous and completed filtration, $\{X(t)\}_{t\geq 0}$ is a \textit{c\`{a}dl\`{a}g} stochastic process, $\{\theta_t\}_{t\geq 0}$ is the shift operator and $\bfP^x$ denotes the probability measure on $(\Omega,\mF)$ for $x\in \sE$. Let $\mT$ be the family of all $\mF_t$-stopping times. Let $f$ and $g$ be two real-valued Borel measurable functions on $\sE$. Define the objective function $J_x(\tau)$ by
\begin{equation}\label{eqn:payofffunction}
J_x(\tau) := \bfE^x\Big[\int_{0}^{\tau}{e^{-as}f(X(s))\, \mathrm{d}s}+e^{-a\tau}g(X(\tau))\Big]  \text{ for } x\in \sE \text{ and } \tau \in \mT,
\end{equation}
 where $f$ is a running benefit function, $g$ is a terminal reward function and $a>0$ is a constant discount factor.

We consider the following optimal stopping problem: find $\tau^* \in \mathcal{T}$ such that
\begin{equation}\label{eqn:valuefunction}
V(x) := \sup_{\tau\in\mathcal{T}}{J_x(\tau)} =J_x(\tau^\ast),
\end{equation}
for each  $x\in \sE$. Our main goal is to study properties of the value function $V$.

The following assumptions holds throughout this paper.
\begin{assumption} \label{ass:fscompact}\leavevmode
\begin{enumerate}
\item $\sE$ is a locally compact, separable metric space with metric $\rho$.
\item $X:=(\Omega,\mF,\mF_t,X_t,\theta_t,\bfP^x)$ is a Feller process with the state space $(\sE,\mathcal{E})$, which has a Feller semigroup $\{\tmP_t\}_{t\geq 0}$, whose generator is $(\tmL,D(\tmL))$ with a core $(\tmG,D(\tmG))$.
\item $a>0$ and $f,g\in \C_b(\sE)$.
\end{enumerate}
\end{assumption}

\assref{ass:fscompact} does not make any a priori supposition on the partial differential equation satisfied by the generator of the Feller process. We first recall a result on the continuity of the value function $V$ given by \eqnref{eqn:valuefunction}. The proof of the continuity is based on the penalty method which consists in finding a sequence $\{v_\lambda\}_{\lambda >0}$ in $\cz(\sE)$ that converges uniformly to the value function $V$.  More precisely, the penalty function $v_\lambda$ is defined as the solution to the following equation
\begin{equation}\label{eqn:osipenvar}
av_\lambda-\tmL v_\lambda-f = \lambda{(g-v_\lambda)}^+,
\end{equation}
where $\lambda>0$.
The next results which are similar to \cite[Theorem I.2.1 and Theorem I.3.1]{robin1978controle} provide the continuity of the value function.
\begin{theorem}\label{thm:osi}
Under \assref{ass:fscompact},
\begin{enumerate}
\item Equation \eqref{eqn:osipenvar} admits a unique solution $v_\lambda\in D(\tmL )$ for each $\lambda>0$.
\item The value function $V$ defined by \eqref{eqn:valuefunction} is in $\cz(\sE)$. In addition, $\{v_\lambda\}_{\lambda>0}$ defined by \eqref{eqn:osipenvar} converges uniformly to $V$ from below as $\lambda\rightarrow \infty$.
\end{enumerate}
\end{theorem}
\begin{proof}
See Appendix A.
\end{proof}
For more information on the continuity of the value function and its extensions;  readers are referred to \cite{ palczewski2010finite, palczewski2011stopping,  palczewski2014infinite, robin1978controle, stettner2011penalty, stettner1983optimal}).
The optimal stopping time for the above optimal stopping problem is obtained using \cite[Theorem~ I.3.3]{robin1978controle} as follows.
\begin{theorem}\label{thm:obtainost}
Under \assref{ass:fscompact}, the optimal stopping time for problem \eqref{eqn:valuefunction} is
\begin{align}\label{eqn:taustar}
\tau^*:=\inf\{t\geq 0; V(X(t))=g(X(t))\}.
\end{align}
\end{theorem}

Let $(\mA,D(\mA))$ denotes an operator with its domain. Recall that, we wish to study the link between the value function $V$ defined by \eqref{eqn:valuefunction} and the unique viscosity solution associated with  $(\mA,D(\mA))$ to the corresponding Hamilton-Jacob-Bellman (HJB) equation
\begin{equation}\label{eqn:defvs1}
\min{(aw-\mA w-f,w-g)}=0.
\end{equation}
Thus, we first give the definition of viscosity solution:
\begin{definition}(\textsl{Viscosity Solution})\label{def:vs}
Given an operator with domain $(\mA, D(\mA))$, a function $w\in USC(\sE)$ (\resp{$w\in LSC(\sE)$}) is a \textsl{viscosity subsolution} (\resp{\textsl{supersolution}}) associated with $(\mA,D(\mA))$ to \eqnref{eqn:defvs1} if for all $\phi\in D(\mA )$ such that  $\phi-w$ has a global minimum (\resp{maximum}) at $x_0\in\sE$ with $\phi(x_0) = w(x_0)$,
\begin{equation}\label{eqn:defvs1def}
\min{(a\phi(x_0)-\mA \phi(x_0)-f(x_0),\phi(x_0)-g(x_0))} \leq (\geq)0.
\end{equation}
Furthermore, $w\in \C(\sE)$ is a \textsl{viscosity solution} associated with $(\mA,D(\mA))$ to \eqnref{eqn:defvs1} if it is both a viscosity supersolution and a viscosity subsolution.
\end{definition}

Next, let us introduce the notion of $a$-generator:

\begin{definition}(\textsl{$a$-generator})\label{def:agen}
Let $X=(\Omega,\mF,\mF_t,X_t,\theta_t,\bfP^x)$ be a Markov process on the state space $(\sE,\mathcal{E})$. Set $a> 0$. An operator $(\mA,D(\mA))$ is called an \textsl{a-supergenerator} (\resp{\textsl{a-subgenerator}, \textsl{a-generator}}) of $X$, if for any $w\in D(\mA)$, the process $\{S_w(t)\}_{t\geq 0}$ defined by
\begin{equation}\label{eqn:defasup}
S_w(t):=w(X(0))-e^{-at}w(X(t))-\int_0^t e^{-as}(aw-\mA w)(X(s))\diff s,
\end{equation}
is a $(\mF_t,\bfP^x)$ uniformly integrable supermartingale (\resp{submartingale, martingale}) for all $x\in \sE$.
\end{definition}

\section{Existence of viscosity solution}\label{sec:p1}

In this section, we show that the value function defined by \eqnref{eqn:valuefunction} can be described as a viscosity solution associated with the generator $(\tmL,D(\tmL))$ of the Feller process or its core $(\tmG,D(\tmG))$.
We prove that the value function defined by \eqnref{eqn:valuefunction} is a viscosity  supersolution (\resp{subsolution, solution}) associated with an extended generator of the Feller process.

\begin{theorem}\label{thm:vsexist}
Suppose \assref{ass:fscompact} holds. Suppose $(\mA, D(\mA))$ is an $a$-supergenerator \textup{(}\resp{$a$-subgenerator, $a$-generator}\textup{)} of $X$ and $\mA:D(\mA)\subseteq \C(\sE)\rightarrow \C(\sE)$. Then the value function $V$ defined by \eqref{eqn:valuefunction} is a viscosity supersolution \textup{(}\resp{subsolution, solution}\textup{)} associated with $(\mA, D(\mA))$  to
\begin{equation}
\min(aw-\mA w-f,w-g)=0.
\end{equation}
\end{theorem}
\begin{proof}
The method used to show the existence is based on the probabilistic description of the extended generator of the Feller process $\{X(t)\}_{t\geq 0}$. See \secref{sec:proofex} for a detailed proof.
\end{proof}
\begin{remark}
As seeing later, enlarging the domain $D(\mA)$ has the advantage that it allows to exclude functions which are viscosity solutions. Hence, we use the solution of the martingale problem to define the extended generator instead of the infinitesimal generator $(\tmL,D(\tmL))$. The former enables us to provide more choices on the test function in $D(\mA)$. For example, $D(\mA)$ can be chosen to be $\C_b^2(\sE)$ or could even include an unbounded function space; see  for instance \secref{exa:levy}.
\end{remark}

One can also show as in \cite[Lemma~2.9]{costantini2015} that if the process $\{S^{(0)}_w(t)\}_{t\geq 0}$ defined by $$S^{(0)}_w(t):=w(X_0)-w(X_t)+\int_0^t \mA w(X(s))\diff s$$
is a $(\mF_t,\bfP^x)$-martingale for any $x\in \sE$ and $w\in D(\mA)$, then $(\mA,D(\mA))$ is an $a$-generator for $a>0$ when $\mA:D(\mA)\subseteq B(\sE)\rightarrow B(\sE)$.  Therefore, by Dynkin's formula, the infinitesimal generator $(\tmL,D(\tmL))$ of the Feller process or its core $(\tmG,D(\tmG))$  is an $a$-generator for all $a>0$. In this case, \thmref{thm:vsexist} implies the following corollary:

\begin{corollary}\label{cor:lviscositysolutionexists}
Under \assref{ass:fscompact}, the value function $V$ defined by \eqnref{eqn:valuefunction} is a viscosity solution associated with the infinitesimal generator $(\tmL,D(\tmL))$ to
\begin{equation}\label{eqn:lviscositysolution}
\min{(aw-\tmL w-f,w-g)}=0.
\end{equation}
\end{corollary}

The proofs of the above results is given by the following section.

\subsection{Proof of \thmref{thm:vsexist}}\label{sec:proofex}
This section is devoted to the proof of \thmref{thm:vsexist}. The proof is standard with a modification due to the presence of the absorbing state. The proof will be given for two classes of the initial state $x\in E$: the absorbing and the non-absorbing states.

We say that $x\in \sE$ is an absorbing state if and only if $X_t = x$ for all $t\in [0,\infty)$ almost surely under $\bfP^{x}$. Let $\tau_\delta$ be an $\mF_t$-stopping time  defined by
\begin{align}\label{eqn:taudelta}
\tau_\delta := \inf\{s\geq 0; X(s)\not\in \bar{B}(X(0),\delta)\},
\end{align}
where $\delta>0$ and ${B}(x,\delta):=\{y\in \sE; \; \rho(x,y)< \delta\}$.
 The following lemma that can be found in \cite[Lemma 17.22]{kallenberg2006foundations} provides information on the stopping time $\tau_\delta$ when the initial state $x$ is absorbing or not.
\begin{lemma} \label{thm:absorbingcharacterics} Let $X$ be a Feller process.
\begin{enumerate}
 \item Assume $x\in\sE$ is not absorbing. Then $\bfE^x[\tau_\delta]<\infty$ for all sufficiently small $\delta>0$,
 \item $x\in \sE$ is absorbing if and only if $\bfP(\tau_\delta=\infty)=1$ for all $\delta>0$.
\end{enumerate}
\end{lemma}

The subsequent lemmas are needed in the proof of the existence of the viscosity solution for absorbing initial state process. Their proofs are standard. However for the sake of completeness, we provide details.

\begin{lemma}\label{lem:valueabosrbed}
Suppose \assref{ass:fscompact} holds. Suppose in addition that the initial state $x\in\sE$ is absorbing. Then the value function satisfies
\begin{equation}\label{eqn:valueabosrbed}
V(x) = \max\Big(\frac{f(x)}{a},g(x)\Big).
\end{equation}
\end{lemma}
\begin{proof}
Since the initial state $x \in \sE$ of the Feller process $X$ is absorbing, we have $X_t=x$ for all $t\in [0,\infty)$ $\bfP^{x}$-a.s. For $x\in \sE$,
\begin{align*}
V(x) =& \sup_{\tau}\bfE^x\Big[\int_0^\tau e^{-as}f(x)\diff s+e^{-a\tau}g(x)\Big]\\
=& \sup_{\tau}\bfE^x\Big[\frac{f(x)(1-e^{-a\tau})}{a}+e^{-a\tau}g(x)\Big]\\
=& \sup_{\tau}\bfE^x\Big[\frac{f(x)}{a}+e^{-a\tau}\Big(g(x)-\frac{f(x)}{a}\Big)\Big].
\end{align*}
If $g(x)>\frac{f(x)}{a}$, then $V(x)\leq g(x)$ and the equality is attained on the set $\{\tau=0\}$, that is, $V(x) = g(x)$, otherwise, $V(x)\leq \frac{f(x)}{a}$ and the equality is attained on the set $\{\tau=\infty\}$, that is, $V(x) = \frac{f(x)}{a}$.
\end{proof}

\begin{lemma}\label{lem:osiex1}
Suppose \assref{ass:fscompact} holds. For $x\in \sE$ and $\delta>0$,
\begin{equation}\label{eqn:valuefunctionlg2}
V(x) \geq\bfE^{x}\Big[\int_{0}^{\tau_\delta}{e^{-as}f(X(s))\diff s} +e^{-a\tau_\delta}V(X(\tau_\delta))\Big].
\end{equation}
Suppose in addition that $V(x)>g(x)$. Then there exists a constant $\Delta>0$ such that for any $\delta\leq \Delta$, we have
\begin{equation}\label{eqn:valuefunctionlg}
V(x) = \bfE^{x}\Big[\int_{0}^{\tau_\delta}{e^{-as}f(X(s))\diff s}+e^{-a\tau_\delta}V(X(\tau_\delta))\Big].
\end{equation}
\end{lemma}

\begin{proof}
Let $Z(t):= \int_0^t e^{-as}f(X(s))\diff s + e^{-a t} V(X(t))$ for $t\geq 0$. Then, using Snell envelope (see for example \cite[Theorem 2.4]{peskir2006optimal}), the process $\{Z_t\}_{t\geq 0}$ is a supermartingale and $\{Z_t\}_{t\wedge \tau ^*}$ is a martingale, where $\tau^*$ is defined by \eqref{eqn:taustar}. Therefore, \eqref{eqn:valuefunctionlg2} and \eqref{eqn:valuefunctionlg} follows. In particular, \eqref{eqn:valuefunctionlg} follows from the fact that $\sE$ is a separable metric space and $V$ and $g$ are continuous.
\end{proof}
We now turn to the proof of \thmref{thm:vsexist}.
\begin{proof}[Proof of \thmref{thm:vsexist}] \leavevmode
	
\textbf{(1)} \textbf{Viscosity Supersolution:} Suppose $(\mA,D(\mA))$ is an $a$-supergenerator of the Markov process $X$. Suppose  $x\in \sE$ and $\phi\in D(\mA )$ such that $\phi(x) = V(x)$ and $\phi -V$ has a global maximum at $x\in\sE$. We wish to prove that
$$\min(a\phi(x)-\mA \phi(x)-f(x),\phi(x)-g(x))\geq 0.$$
Since $\phi(x)=V(x) \geq g(x)$, it is sufficient to prove that
\begin{equation}\label{eqn:existencesuper}
a\phi(x)-\mA \phi(x)-f(x)\geq 0.
\end{equation}
\textbf{Case 1.} Assume that $x\in \sE$ is an absorbing initial state, that is $X_t = x$ for all $t\in [0,\infty)$ $\bfP^x$-a.s. 
and define the process $\{S_\phi(t)\}_{t\geq 0}$ by
\begin{eqnarray*}
S_\phi(t):=\phi(X_0)-e^{-at}\phi(X_t)-\int_0^t e^{-as}(a\phi-\mA \phi)(X(s))\diff s.
\end{eqnarray*}
Since $x\in\sE$ is absorbing, $S_\phi(t) = \int_0^t e^{-as}\mA \phi(x)\diff s$ for all $t\in [0,\infty)$ $\bfP^x$-a.s. Since $(\mA,D(\mA))$ is an $a$-supergenerator, it follows that  $\{S_\phi(t)\}_{t\geq 0}$ is a $(\mF_t,\bfP^x)$ uniformly integrable supermartingale, and therefore $\mA\phi(x)\leq 0$. 
In addition, using \lemref{lem:valueabosrbed}, we have $\phi(x)=V(x) = \max(f(x)/a,g(x))$. The latter combines with the fact that $\mA \phi(x)\leq 0$ yields \eqnref{eqn:existencesuper}. \\
\textbf{Case 2.} Assume that $x\in \sE$ is not an absorbing initial value. It follows from \lemref{thm:absorbingcharacterics} that $\bfE^{x}[\tau_\delta]< \infty$ for all small enough $\delta>0$. Since $\phi\in D(\mA)$ and $\phi(y)-V(y)\leq 0$ for any $y\in\sE$, \eqnref{eqn:valuefunctionlg2} implies that
\begin{align}\label{eqvnat1}
V(x)\geq& \bfE^{x}\Big[\int_0^{\tau_\delta}e^{-as}f(X(s))\diff s+e^{-a\tau_\delta}V(X(\tau_\delta))\Big]\notag\\
\geq&\bfE^{x}\Big[\int_0^{\tau_\delta}e^{-as}f(X(s))\diff s+e^{-a\tau_\delta}\phi(X(\tau_\delta))\Big]\notag \\
\geq &\bfE^{x}\Big[\int_0^{\tau_\delta}e^{-as}(f(X(s))+\mA \phi(X(s))-a\phi{(X(s))})\diff s\Big]+\phi(x),
\end{align}
where the last inequality follows from the optional sampling theorem since $(\mA,D(\mA))$ is an $a$-supergenerator.
Since $V(x) = \phi(x)$, dividing both sides of \eqref{eqvnat1} by $\bfE^{x}\big[\tau_\delta\big]$, we obtain
\begin{align}\label{eqvnat2}
0&\geq \frac{\bfE^{x}\big[\int_0^{\tau_\delta}e^{-as}(f(X(s))+\mA \phi(X(s))-a\phi{(X(s))})\diff s\big]}{\bfE^{x}\big[\tau_\delta\big]}\nonumber\\
&\geq \frac{\bfE^{x}\big[\int_0^{\tau_\delta}e^{-as}C_-(x,\delta)\diff s\big]}{\bfE^{x}\big[\tau_\delta\big]}\nonumber\\
&= \frac{1-\bfE^{x}[e^{-a\tau_\delta}]}{\bfE^{x}\big[\tau_\delta\big]}C_-(x,\delta),
\end{align}
where $C_-(x,\delta)=\inf_{y\in B(x,\delta)}(f(y)+\mA \phi(y)-a\phi{(y)})$. Since $\bfE^{x}\big[\tau_\delta\big]$ is bounded such that $\frac{1-\bfE^{x}[e^{-a\tau_\delta}]}{\bfE^{x}\big[\tau_\delta\big]}>0$,  \eqref{eqvnat2} yields $C_-(x,\delta)\leq 0$ for all $\delta>0$. Since $f, \mA\phi$ and $\phi$ are in $\C(\sE)$, $C_-(x,\delta)$ converges pointwise to $f(x)+\mA \phi(x)-a\phi{(x)}$ as $\delta\rightarrow 0$. Hence, \eqnref{eqn:existencesuper} is proved.

\textbf{(2)} \textbf{Viscosity Subsolution:} Assume that $(\mA,D(\mA))$ is an $a$-subgenerator of the process $X$. Choose $\psi\in D(\mA)$, such that $\psi(x )=V(x)$ and $\psi-V$ has a global minimum at $x$. If $V(x)=g(x)$, we find a viscosity subsolution by setting $\psi(x)=V(x)=g(x)$. Since $V\geq g$, we thus only consider the initial state $x\in\sE$ satisfying $V(x)-g(x)>0$. Hence, it is enough to show that
\begin{equation}\label{eqn:existencesub}
a\psi(x)-\mA \psi(x)-f(x)\leq 0.
\end{equation}

\noindent \textbf{Case 1.} Assume that $x\in\sE$ is absorbing. Then by \lemref{lem:valueabosrbed}, $\psi(x) = V(x)=f(x)/a$. Since $(\mA,D(\mA))$ is an $a$-subgenerator, applying similar arguments as in the proof of Case 1 for viscosity supersolution,  we obtain $\mA \psi(x) \geq 0$. Therefore, \eqnref{eqn:existencesub} is satisfied.\\ 
\textbf{Case 2.} Assume that $x$ is not absorbing. Then by \lemref{thm:absorbingcharacterics} and \lemref{lem:osiex1}, there exists a constant $\Delta>0$ such that for any $\delta \leq \Delta$, we have $\bfE^{x}[\tau_\delta]<\infty$ and \eqnref{eqn:valuefunctionlg} holds. Since $\psi\in D(\mA)$ and $\psi\geq V$, we have
\begin{align}\label{eqvnatsb1}
V(x) =& \bfE^{x}\Big[\int_{0}^{\tau_\delta}e^{-as}f(X(s))\diff s+e^{-a{\tau_\delta}}V(X({\tau_\delta}))\Big]\notag \\
\leq& \bfE^{x}\Big[\int_{0}^{\tau_\delta}e^{-as}f(X(s))\diff s+e^{-a{\tau_\delta}}\psi(X({\tau_\delta}))\Big]\notag \\
\leq& \psi(x)+\bfE^{x}\Big[\int_{0}^{\tau_\delta}e^{-as}\big(f(X(s))+\mA \psi(X(s))-a\psi(X(s))\big)\diff s\Big],
\end{align}
where the last inequality follows from the optional sampling theorem since $(\mA,D(\mA))$ is an $a$-subgenerator.
Since $V(x)=\psi (x)$, dividing $\bfE^{x}\big[\tau_\delta\big]$ on both sides of \eqref{eqvnatsb1}, we get
\begin{equation}\label{eqn:existencesub1}
0 \leq \frac{\bfE^{x}\Big[\int_{0}^{\tau_\delta}e^{-as}\big(f(X(s))+\mA \psi(X(s))-a\psi(X(s))\big)\diff s\Big]}{\bfE^{x}\big[\tau_\delta\big]},
\end{equation}
for all $\delta\leq \Delta$. Then, since $f, \mA\phi$ and $\phi$ belong to $\C(\sE)$ and $\bfE^{x}\big[\tau_\delta\big]$ is bounded, taking $\delta \rightarrow 0$, we obtain the desired result.
\end{proof}

\section{Uniqueness of viscosity solution for compact state space $\sE$}\label{uninoncomp}
In this section, we prove that the value function is the uniqueness of viscosity solution under the assumption that the state space is compact. \thmref{thm:comparisonprinciple} gives a comparison principle for viscosity supersolution and subsolution which is needed in the proof of the uniqueness of the viscosity solution associated with $(\tmL,D(\tmL))$ (\resp{$(\tmG,D(\tmG))$}) (see \thmref{thm:gvs=V}).

\begin{theorem}{(Comparison Principle)}\label{thm:comparisonprinciple}. Suppose \assref{ass:fscompact} holds. Suppose $\sE$ is compact and the constant function $1\in D(\tmG)$. Furthermore, suppose $w_1$ is a viscosity supersolution and $w_2$ is a viscosity subsolution associated with $(\tmG,D(\tmG))$ to
\begin{equation}\label{eqn:gviscositysolution}
\min{(aw-\tmG w-f,w-g)}=0.
\end{equation}
Then $w_1\geq w_2$.
\end{theorem}
\begin{proof}
See \secref{sec:proofun}.
\end{proof}
\begin{remark}
Since we suppose in the proof that all constant functions belong to $D(\tmL)$ or $D(\tmG)$, it is natural to assume the compactness of $\sE$. However, the latter is  not a necessary condition to show the uniqueness of the viscosity solution and will be relaxed in the subsequent sections; see for example \secref{sec:notc1} and \propref{prop:g*vs}.
\end{remark}

The following theorem constitutes the second main result of this section

\begin{theorem}\label{thm:gvs=V}
Suppose \assref{ass:fscompact} holds. Suppose $\sE$ is compact and the constant function $1\in D(\tmG)$. Then the value function \eqref{eqn:valuefunction} is the unique viscosity solution associated with  $(\tmG,D(\tmG))$ to \eqref{eqn:gviscositysolution}.
\end{theorem}
\begin{proof}
The existence follows from \corref{cor:lviscositysolutionexists}. Using \thmref{thm:comparisonprinciple}, if there exists another viscosity solution, it must coincide with the value function.
\end{proof}

\subsection{Proof of \thmref{thm:comparisonprinciple}}\label{sec:proofun}

We first recall that the state space $\sE$ is compact, $D(\tmG)$ (or $D(\tmL)$) contains constant functions and $(\tmG, D(\tmG))$ is the core of the infinitesimal generator $(\tmL, D(\tmL))$. We prove \thmref{thm:comparisonprinciple} in three steps. In the first step, we define a notion of classical solution to \eqnref{eqn:gviscositysolution} and show a partial comparison principle between a classical subsolution (\resp{supersolution}) and a viscosity supersolution (\resp{subsolution}). Second, we show that there exists a sequence of classical subsolutions (\resp{supersolutions}) that converges from below (\resp{above}) to the value function $V$ defined by \eqnref{eqn:valuefunction}. Finally, we use the results from steps 1 and 2 to prove \thmref{thm:comparisonprinciple}.

\subsubsection*{\textbf{Step 1}} In this step, we first define the notion of classical subsolution (\resp{supersolution}) to \eqnref{eqn:gviscositysolution} and then prove a classical comparison theorem.
\begin{definition}\label{def:l classical solution}
A function $w$ is a classical subsolution (\resp{supersolution}) associated with $(\mA,D(\mA))$ to \eqnref{eqn:gviscositysolution}, if $w\in D(\mA )$ and satisfies
\begin{equation}
\min{(aw-\mA w-f,w-g)} \leq (\geq)0.
\end{equation}
\end{definition}
\begin{lemma}\label{lem:classical is viscosity}
Suppose \assref{ass:fscompact} holds. Then a classical subsolution \textup{(}\resp{supersolution}\textup{)} associated with $(\tmG,D(\tmG))$ to \eqref{eqn:gviscositysolution} is a viscosity subsolution \textup{(}\resp{supersolution}\textup{)} associated with $(\tmG,D(\tmG))$ to \eqref{eqn:gviscositysolution}.
\end{lemma}
\begin{proof}
Let $w$ be a classical subsolution to \eqnref{eqn:gviscositysolution}. By contradiction, assume that $w$ is not a viscosity subsolution to \eqnref{eqn:gviscositysolution}. Then, there exists a function $\phi\in D(\tmG )$ such that $\phi-w$ has a global minimum at $x_0$ with $(\phi-w)(x_0) = 0$ and
\begin{equation}\label{eqn:proof classical is viscosity}
\min{(a\phi(x_0)-\tmG \phi(x_0)-f(x_0),w(x_0)-g(x_0))} > 0.
\end{equation}
Since $w-\phi$ has a global nonnegative maximum at $x_0$, the positive maximum principle yields $\tmG (w-\phi)(x_0)\leq0$, that is, $\tmG w(x_0)\leq \tmG \phi(x_0)$. This together with $w(x_0)=\phi(x_0)$ and \eqnref{eqn:proof classical is viscosity} gives
$$\min{(a w(x_0)-\tmG w(x_0)-f(x_0),w(x_0)-g(x_0))} > 0,$$
hence contradicting the assumption that $w$ is a classical subsolution to \eqnref{eqn:gviscositysolution}. Therefore $w$ is a viscosity subsolution to \eqnref{eqn:gviscositysolution}.
The proof for the supersolution follows in the same way.
\end{proof}

We will also need the following partial comparison theorem.
\begin{lemma}[{\textsl{Partial Comparison Principle}}]\label{lem:partialcomparisonprinciple} Suppose that conditions of Theorem \ref{thm:gvs=V} hold. Let $w_1$ be a supersolution associated with $(\tmG,D(\tmG))$ to \eqref{eqn:gviscositysolution} and $w_2$ be a subsolution associated with $(\tmG,D(\tmG))$ to \eqref{eqn:gviscositysolution}, where one of the solutions is in the classical sense and the other is in the viscosity sense. Then, $w_1\geq w_2$.
\end{lemma}
\begin{proof}
Let $w_1$ be a classical supersolution to \eqnref{eqn:gviscositysolution} and $w_2$ be a viscosity subsolution to \eqnref{eqn:gviscositysolution}. Since $D(\tmG)\subseteq \cz(\sE)$, we have that $w_1 \in \cz(\sE)$. Since $w_2\in USC(\sE)$ and $\sE$ is compact, there exists $x_0\in \sE$ such that
$$
\delta := \sup_{y\in\sE}(w_2-w_1)(y)=(w_2-w_1)(x_0).
$$
 By contradiction, assume that $\delta >0$ and define $w_1^* $ by
 $$w_1^* := w_1+\delta.$$
  Since $w_1, \delta \text{ (as a constant function) are in }D(\tmG )$ and $w_1^*-w_2$ has a global minimum at $x_0$ with  $(w_1^*-w_2)(x_0)=0$, it follows that $w_1^*$ is a well defined test function for the viscosity subsolution $w_2$. Moreover, by the positive maximum principle, we have $\tmG \delta\leq 0$. Hence,
\begin{align*}\label{eqn:osipcppglobal}
\min(aw_1^*-\tmG w_1^*-f,w_1^*-g)(x_0)
=&\min(a(w_1+\delta)-\tmG w_1-\tmG\delta-f,w_1+\delta-g)(x_0) \\
\geq& \min(aw_1-\tmG w_1-f,w_1-g)(x_0)+\min(a,1)\delta.
\end{align*}
Since $w_1$ is a classical supersolution, we have
$$\min(aw_1-\tmG w_1-f,w_1-g)(x_0)\geq 0.$$
 Therefore,
$$\min(aw_1^*-\tmG w_1^*-f,w_1^*-g)(x_0)\geq \min(a,1)\delta>0.$$
This contradicts the fact that $w_2$ is a viscosity subsolution to \eqnref{eqn:gviscositysolution}. Thus $\sup_{x\in\sE}(w_2-w_1)(x_0)=\delta\leq0$, that is, $w_2\leq w_1$ on $\sE$.
Similar arguments can be used to show that $w_1\geq w_2$, if $w_1$ is a viscosity supersolution to \eqnref{eqn:gviscositysolution} and $w_2$ is a classical subsolution to \eqnref{eqn:gviscositysolution}.
\end{proof}

\begin{corollary}{(\textsl{Classical Comparison Principle})}\label{lem:classicalcomparisonprinciple}
Suppose that conditions of Theorem \ref{thm:gvs=V} hold. Let $w_1$ be a classical supersolution and $w_2$ be a classical subsolution associated with $(\tmG,D(\tmG))$ to \eqnref{eqn:gviscositysolution}. Then, $w_1\geq w_2$.
\end{corollary}
\begin{proof}
By \lemref{lem:classical is viscosity}, we know that a classical supersolution (\resp{subsolution}) to \eqnref{eqn:gviscositysolution} is also a viscosity supersolution (\resp{subsolution}) to \eqnref{eqn:gviscositysolution}. Then, by partial comparison principle, the result follows.
\end{proof}

\subsubsection*{\textbf{Step 2}} We first show that there exists a sequence of classical supersolution (\resp{subsolution}) that converges from above (\resp{below}) to the value function $V$.

\begin{lemma}\label{lem:lcomparison}
Suppose that conditions of Theorem \ref{thm:gvs=V} hold. Then there exists a sequence of  classical supersolutions (\resp{subsolutions}) associated with $(\tmL,D(\tmL))$ to \eqref{eqn:lviscositysolution} that converges to the value function $V$ defined by \eqref{eqn:valuefunction} uniformly from the above (\resp{below}).
\end{lemma}
\begin{proof}\leavevmode

\textbf{(1)} \textbf{Classical Supersolutions.} It follows from \thmref{thm:osi} that the sequence $\{v_{\lambda}\}_{\lambda> 0}\in D{(\tmL )}$ defined by \eqnref{eqn:osipenvar}  converges uniformly to $V$ from below when $\lambda\rightarrow \infty$. Thus, there exists a subsequence $\{\lambda_n\}_{n\in\mN}$ such that $0\leq V-v_{\lambda_n} \leq \frac{1}{n}.$ Define the sequence $\{w_n\}_{n\in\mN}$ by
 $$w_n := v_{\lambda_n}+\frac{1}{n} \text{ for } n\in\mN.$$
  Then for $n\in\mN$
\begin{equation}\label{eqn:ccpproof1}
0\leq w_n- V= v_{\lambda_n}+\frac{1}{n}-V\leq \frac{1}{n}.
 \end{equation}
 Combining \eqnref{eqn:osipenvar} and \eqref{eqn:ccpproof1} and using the fact that $\tmL(1/n)\leq 0$ (positive maximum principle), we obtain
\begin{align*}
aw_n-\tmL w_n-f =& a\left(v_{\lambda_n}+\frac{1}{n}\right)-\tmL\left(v_{\lambda_n}+\frac{1}{n}\right)-f\\
=& av_{\lambda_n}-\tmL v_{\lambda_n}-f+\frac{a}{n}-\tmL\frac{1}{n}\\
\geq & \lambda_n(g-v_{\lambda_n})^+ +\frac{a}{n}>0.
\end{align*}
Since $w_n-g\geq w_n-V\geq 0$ by \eqref{eqn:ccpproof1}, the above inequalities imply $\min(aw_n-\tmL w_n-f,w_n-g)\geq 0$, that is, $w_n\in D(\tmL)$ is a classical supersolution to \eqnref{eqn:lviscositysolution}. Furthermore, by \eqnref{eqn:ccpproof1}, $\{w_n\}_{n\in \mN}$ is a sequence of classical supersolutions to \eqnref{eqn:lviscositysolution} that  converges  uniformly to $V$ from above as $n\rightarrow \infty$.

\textbf{(2)} \textbf{Classical Subsolutions.}  Choose once more the sequence $\{v_{\lambda}\}_{\lambda> 0}$ defined by \eqnref{eqn:osipenvar}. For  any $\lambda>0$ or $x\in \sE$, one of the following two expressions $v_\lambda(x)-g(x)$ and $\lambda (g(x)-v_\lambda(x))^+$ is non-positive. Then,
$$\min(av_\lambda-\tmL v_\lambda-f,v_\lambda-g)(x) = \min(\lambda(g-v_\lambda)^+,v_\lambda-g)\leq 0.$$
Hence, $\{v_\lambda\}_{\lambda>0}$ is a sequence of classical subsolutions to  \eqnref{eqn:lviscositysolution} and its uniform convergence from below becomes straightforward by \thmref{thm:osi}.
\end{proof}

\begin{corollary}\label{cor:gcomparison}
Suppose that conditions of Theorem \ref{thm:gvs=V} are in force. Then there exists a sequence of classical  supersolutions (\resp{subsolutions}) associated with $(\tmG,D(\tmG))$ to \eqref{eqn:gviscositysolution} that converges to the value function $V$ defined by \eqref{eqn:valuefunction} uniformly from above (\resp{below}).
\end{corollary}
\begin{proof}
We know from \lemref{lem:lcomparison} that there exists a sequence of classical supersolutions associated with $(\tmL, D(\tmL))$ to \eqnref{eqn:lviscositysolution} such that $\{w_n\}_{n\in \mN}$ satisfies $0\leq w_n-V\leq 1/n$ for $n\in \mN$.
Let $\varepsilon>0$ and choose an integer $n_0$ such that $n_0\geq \frac{4}{\varepsilon}$. Set $$w^{(\varepsilon)}:=w_{n_0}+\varepsilon/4.$$
Then,
\begin{align}\label{eqn:prooflgtransform234}
 w^{(\varepsilon)}-V =& (w^{(\varepsilon)}-w_{n_0})+(w_{n_0} - V)\leq \frac{\varepsilon}{4} + \frac{1}{n_0} \leq \frac{\varepsilon}{2},\\
\mbox{and }\quad w^{(\varepsilon)}-V=& w_{n_0}+ \frac{\varepsilon}{4}-V \geq \frac{\varepsilon}{4}.
\end{align}

Since $w_{n_0}$ is  a classical supersolution to \eqnref{eqn:lviscositysolution}, we have
	\begin{align}\label{eqn:prooflgtransform134}
	 \min(aw^{(\varepsilon)}-\tmL w^{(\varepsilon)}-f,w^{(\varepsilon)}-g)
	=&\min(a(w_{n_0}+\varepsilon/4)-\tmL(w_{n_0}+\varepsilon/4)-f,w_{n_0}+\varepsilon/4-g)\notag  \\
	\geq&\frac{ \min(a,1)\varepsilon}{4}+\min(a w_{n_0}-\tmL w_{n_0}-f,w_{n_0}-g)\notag \\
	\geq&\frac{ \min(a,1)\varepsilon}{4}.
	\end{align}
Therefore, $w^{(\varepsilon)}$ is also a classical supersolution associated with $(\tmL,D(\tmL))$ to \eqnref{eqn:lviscositysolution}. Since $(\tmG,D(\tmG))$ is the core of $(\tmL,D(\tmL))$, it follows that for $w^{(\varepsilon)}\in D(\tmL)$, there exists a sequence $\{u^{(\varepsilon)}_m\}_{m\in \mN}$ in $D(\tmG)$ such that
\begin{eqnarray}\label{eqn:prooflgtransform12313}
\ninf{u^{(\varepsilon)}_m-w^{(\varepsilon)}}\leq \frac{1}{m} \;\mbox{ and }\; \ninf{\tmG u^{(\varepsilon)}_m- \tmL w^{(\varepsilon)}}\leq \frac{1}{m} \mbox{ for any } m\in\mN.
\end{eqnarray}
In the following, we will construct a sequence $\{u_\varepsilon\}_{\varepsilon>0}$ of classical supersolution associated with $(\tmG,D(\tmG))$ to \eqref{eqn:gviscositysolution} that converges to $V$ from above. Since $u^{(\varepsilon)}_m\geq w^{(\varepsilon)}-1/m$ and $\tmG u^{(\varepsilon)}_m\leq \tmL w^{(\varepsilon)}+1/m$, we have
\begin{align}\label{eqn:prooflgtransform12134}
 \min(au^{(\varepsilon)}_m-\tmG u^{(\varepsilon)}_m-f, u^{(\varepsilon)}_m -g)
\geq & \min(a(w^{(\varepsilon)}-\frac{1}{m})-(\tmL w^{(\varepsilon)}+\frac{1}{m})-f, w^{(\varepsilon)}-\frac{1}{m} -g)\notag\\
 \geq & -(a+1)\frac{1}{m} +  \min(aw^{(\varepsilon)}-\tmL w-f,w^{(\varepsilon)}-g).
\end{align}

Choose $m_0:=m_0(\varepsilon) \in \mN$ such that $m_0\geq \max(\frac{4}{\varepsilon}, \frac{4 (a+1)}{\min(a,1)\varepsilon})$, then for $m\geq m_0$, we have: on the one hand, using \eqnref{eqn:prooflgtransform134} and \eqnref{eqn:prooflgtransform12134},  $u_{m}^{(\varepsilon)}$ is a classical supersolution to \eqnref{eqn:gviscositysolution}; on the other hand, using \eqnref{eqn:prooflgtransform234} and \eqnref{eqn:prooflgtransform12313}
\begin{eqnarray*}
0\leq u^{(\varepsilon)}_m-V= (u^{(\varepsilon)}_m-w)+(w-V)\leq \frac{1}{m}+\frac{\varepsilon}{2}\leq \varepsilon.
\end{eqnarray*}
Define a new sequence $\{u_\varepsilon\}_{\varepsilon>0}$ by setting $u_{\varepsilon}:=u^{(\varepsilon)}_{m_0(\varepsilon)}$. Then $u_\varepsilon$  is a classical supersolution associated with $(\tmG,D(\tmG))$ to \eqnref{eqn:gviscositysolution} satisfying $0\leq u_\varepsilon-V\leq \varepsilon$ for any arbitrary $\varepsilon>0$. Therefore, $\{u_\varepsilon\}_{\varepsilon>0}$ converges uniformly to the value function $V$ from above as $\varepsilon\rightarrow 0$.

The case of subsolutions can be proved in a similar way.
\end{proof}
\subsubsection*{\textbf{Step 3}} Finally, we prove the comparison principle stated in \thmref{thm:comparisonprinciple}.
\begin{proof}[Proof of \thmref{thm:comparisonprinciple}]
Define the sets of classical supersolutions and subsolutions associated with $(\tmG,D(\tmG))$ to  \eqnref{eqn:gviscositysolution} as follows,
\begin{align}
H_{sup}&:= \left\{u\in D(\tmG );u \hbox{ is a classical supersolution associated with }(\tmG,D(\tmG))\hbox{ to } \eqnref{eqn:gviscositysolution} \right\} \label{eqn:domain12l}\\
H_{sub}&:= \left\{v\in D(\tmG ); v\hbox{ is a classical subsolution associated with }(\tmG,D(\tmG))\hbox{ to } \eqnref{eqn:gviscositysolution} \right\}.\label{eqn:domain12l11}
\end{align}

Let $w_1$ be a viscosity supersolution associated with $(\tmG,D(\tmG))$ to \eqnref{eqn:gviscositysolution}. By \lemref{lem:partialcomparisonprinciple}, it is true that $w_1\geq u$ for any $u\in H_{sub}$, and then $w_1(x) \geq \sup_{v\in H_{sub}}v(x)$. Similarly,  let $w_2$ be a viscosity subsolution associated with $(\tmG,D(\tmG))$ to \eqnref{eqn:gviscositysolution}, then, $w_2\leq u$  for any $u\in H_{sup}$ and $w_2(x) \leq \inf_{u\in H_{sup}}u(x)$. By \corref{cor:gcomparison}, there exists a sequence of classical supersolutoins $\{u_n\}_{n\in \mN}$ (\resp{subsolutions $\{v_n\}_{n\in \mN}$}) associated with  $(\tmG,D(\tmG))$ to \eqnref{eqn:gviscositysolution} converging uniformly to the value function $V$  from  above (\resp{below}) as $n\rightarrow \infty$. Then for any $x\in \sE$, we have
\begin{eqnarray*}
w_1(x)&\geq&\sup_{v\in H_{sub}}v(x)\geq \limsup\limits_{n\rightarrow\infty}v_n(x)=V(x),\\
w_2(x)&\leq&\inf_{u\in H_{sup}}u(x)\leq \liminf\limits_{n\rightarrow\infty}u_n(x)=V(x).
\end{eqnarray*}
Therefore, $w_1\geq V\geq w_2$. The proof is completed.
\end{proof}

\section{Uniqueness of Viscosity Solution for noncompact state space}\label{sec:notc1}

Both \assref{ass:fscompact} and compactness condition in \thmref{thm:gvs=V} give sufficient conditions to prove the existence and uniqueness of the viscosity solution using probabilistic and analytical techniques. However, the compactness of $\sE$ is not always satisfied for some interesting Feller processes used in practice, for example L\'{e}vy processes on $\R^n$ and one dimension diffusions on $[0,\infty)$; see \secref{exa:levy} and \secref{exa:onedimdiff}. Thus, \thmref{thm:gvs=V} is not immediately applicable for such processes. In addition, since a Feller semigroup is not necessarily conservative, its generator $(\tmL,D(\tmL))$ may not have a corresponding Feller process $X$. In this section, we do not assume the existence of a Feller process 
(confer conditions (2) and (3) in \assref{ass:fscompact}) and neither do we assume the compactness of $\sE$. We first extend the given Feller semigroup on $\cz(\sE)$ to a conservative Feller semigroup on $\C(\sO)$. From this we construct an associated Feller process with the aim of characterizing a viscosity solution associated with a core $(\tmG,D(\tmG))$ of any infinitesimal generator.

Recall that $\sO:=\sE\cup\{\partial\}$ is the one point compactification of $\sE$. We now extend the Feller semigroup $\{\tmP_t\}_{t\geq 0}$ on $\cz(\sE)$ to a semigroup $\{\mP_t\}_{t\geq 0}$ on $\C(\sO)$ defined by
\begin{eqnarray}\label{eqn:extfs}
\mP_t w(x) :=\left\{
                        \begin{array}{ll}
                        \tmP_t (w-w(\partial))|_\sE(x)+w(\partial)&\hbox{ for any } x\in \sE,\\
                        w(\partial) & \hbox{ otherwise,}
                        \end{array}
                 \right.
\end{eqnarray}
where $w\in \C(\sO)$ and $t\geq 0$. Here $f|_\sE$ is the restriction of the function $f$ on $\sE$. It follows from  \cite[Lemmas 17.13 and 17.14]{kallenberg2006foundations} that $\{\ti{\tmP}_t\}_{t\geq 0}$ is a conservative Feller semigroup. Furthermore by \cite[Theorem I.9.4]{blumenthal2007markov}, for a conservative Feller semigroup, there always exists a Feller process $X = (\Omega,\mF,\mF_t,X_t,\theta_t,\bfP^x)$ on the state space $(\sO,\mathcal{E}_\partial)$ such that
\begin{eqnarray}\label{eqn:conservativefeller}
\mP_t w(x) := \bfE^x\left[w(X_t)\right] \quad\mbox{ for } w\in B(\sE) \mbox{ and } x \in \sO.
\end{eqnarray}
This enables to link any Feller semigroup on $\cz(\sE)$ with a Feller process whose state space $\sO$  is the one-point compactification of $\sE$. Hence, \thmref{thm:gvs=V} could also be useful in this case. We first show the relation between the infinitesimal generator of the Feller semigroup $\{\tmP_t\}_{t\geq 0}$ and that of its extension $\{\mP_t\}_{t\geq 0}$. We recall the definition of $\C_*(\sE)$:
$$\C_* (\sE) := \{w\in \C(\sE); w \mbox{ is converges at infinity}\}.$$
For any $w\in \C_*(\sE)$, $w$ has a continuous extension $\ti{w}$ in $\sO$. Assume that $\sE$ is not compact, then by one-point compactification technique, $\sE$ is a dense open subset of $\sO$ and $w$ converges to a unique limit $C$ at infinity. Thus, we can define the unique continuous extension $\ti{w}\in \C(\sO)$ of $w\in \C_*(\sE)$ by 
\begin{equation}\label{ext00}
\ti{w}(x) :=\left\{
                        \begin{array}{ll}
                        w(x)&\hbox{ for any } x\in \sE,\\
                        C & \hbox{ for }x=\partial.
                        \end{array}
                 \right.
\end{equation}
If $\sE$ is compact and $\partial$ is an isolated point, we simply define the continuous extension of $w\in \C_*(\sE)$ by
\begin{equation}\label{ext1}
\ti{w}(x) :=\left\{
                        \begin{array}{ll}
                        w(x)&\hbox{ for } x\in \sE,\\
                        0 &  \hbox{ for }x=\partial.
                        \end{array}
                 \right.
\end{equation}

\subsection{Main Results} In this section, we present the main results. We first define the following operator $(\tmG^*, D(\tmG^*))$ defined by
\begin{align}
\label{eqn:g*define}
\begin{split}
D(\tmG^*)&:= \{ u\in \C_*(\sE); u-\ti{u}(\partial) \in D(\tmG)\},\\
\tmG^* u &:= \tmG (u-\ti{u}(\partial)) \mbox{ for each } u\in D(\tmG^*).
\end{split}
\end{align}
When $\sE$ is compact, it follows from \eqref{ext1} that $(\tmG^*,D(\tmG^*))=(\tmG,D(\tmG))$.  The proof of \thmref{thm:maintheorem} relies on \thmref{thm:gvs=V} and the following key result.
\begin{proposition}\label{prop:g*vs}
	Suppose that $(\tmG,D(\tmG))$ is a core of the Feller semigroup $\{\tmP_t\}_{t\geq 0}$, $a>0$ and $f,g\in \C_*(\sE)$ (When $\sE$ is compact, we additionally assume that $1\in D(\tmG)$.) Then there exists a unique function $w\in \C_*(\sE)$ with boundary condition $\ti{w}(\partial) = \max(\ti{f}(\partial),\ti{g}(\partial))$ such that $w$ is a viscosity solution associated with $(\tmG^*,D(\tmG^*))$ to
	\begin{equation}\label{eqn:vsg*}
	\min(a w - \tmG^* w- f, w-g)=0.
	\end{equation}
	Moreover, the extension $\ti{w}\in \C(\sO)$ is the unique viscosity solution associated with $(\mG,D(\mG))$ (defined by \eqref{eqn:tgdefine0}) to
	\begin{equation}\label{eqn:vslt}
	\min(a \ti{w} - \mG \ti{w}- \ti{f}, \ti{w}-\ti{g})=0,
	\end{equation}
	where $(\mG,D(\mG))$ is the core of Feller semigoup $\{\mP_t\}_{t\geq 0}$ on $\C(\sO)$ defined by \eqnref{eqn:extfs}.
\end{proposition}
\begin{proof} See \secref{sec: proofmain1}.
\end{proof}
\begin{remark}
The above proposition is used to show uniqueness of viscosity solution when the generator is given by a infinitesimal generator of a Feller semigroup rather than a Feller process. Let us emphasize that we need not this Feller semigroup to be conservative nor on a compact state space $\sE$; see for example \corref{cor:visckil1}.
	\end{remark}
The main results of this section are the following.
\begin{theorem}\label{thm:maintheorem}
Suppose \assref{ass:fscompact} holds. Then the value function $V$ defined by \eqref{eqn:valuefunction} is the unique viscosity solution $w\in \C_b(\sE)$ associated with  $(\tmG^*,D(\tmG^*))$ to
	\begin{equation}\label{eqn:vsaaa}
	\min(aw-\tmG^* w-f,w-g)=0.
	\end{equation}
\end{theorem}
\begin{proof}
See \secref{sec: proofmain2}.
\end{proof}

\begin{theorem}\label{thm:maintheoremcompariosncb}
Suppose \assref{ass:fscompact} holds. Let $w_1\in USC(\sE)$ and $w_2\in LSC(\sE)$ be the viscosity subsolution and supersolution associated with $(\tmG^*,D(\tmG^*))$ to \eqref{eqn:vsaaa}, respectively. If $w_1$ and $w_2$ are bounded from above and below,  respectively, then, $w_1\leq w_2$.
\end{theorem}



\begin{remark} The operator $(\tmG^*,D(\tmG^*))$ in \thmref{thm:maintheorem} always contains the constant function by construction. If one chooses an operator that does not contain this function, then the uniqueness might not hold as illustrated below.
	\end{remark}
\begin{example}[Non uniqueness of viscosity solution]
Let $X$ be a standard Brownian motion on $\R$ and choose $\left(\frac{1}{2}D_{xx},\cc^\infty(\R)\right)$ as its core. By definition, the domain of this operator does not contain constant functions. Set $f>0\in \cz(\R)$ and $g=0$ in the optimal stopping problem. Then, the value function defined by \eqref{eqn:valuefunction} is reduced to
\begin{equation}\label{eqn:valfunction}
V(x) := \sup_{\tau\in\mathcal{T}}\bfE^x\Big[\int_{0}^{\tau}{e^{-as}f(X(s))\, \mathrm{d}s}\Big]=J_x(\tau^*)=\tmR_a f(x)
\end{equation}
for $x\in\R$ and the optimal stopping time strategy is $\tau^*=\infty$. By \thmref{thm:vsexist}, $V=\tmR_a f\in \cz(\R)$ is a viscosity solution associated with $(\frac{1}{2}D_{xx},\cc^\infty(\R))$ to
$$\min(aw-\frac{1}{2}D_{xx} w-f,w)=0.$$
Let $c>0$ and set $w=c\tmR_a f>0$. We claim that there is no $\phi\in \cc^\infty(\R)$ such that $\phi-w$ has a global minimum equal $0$ at $x_0\in\R $. Indeed assume that there exists $x_0\in \R$ such that
\begin{align}\label{eqn:remarkares}
\phi(x_0)-w(x_0)=0\leq \phi(x)-w(x) \text{ for all }x\in \R.
\end{align}
Since $\phi$ is of compact support, there exists $y_0\in \R$ such that $\phi(y_0)=0$. Choose $x=y_0$ then $\phi(y_0)-w(y_0)=-w(y_0)<0$. This contradict the fact that $\phi-w$ has a global minimum equal $0$ at $x_0$.
  Since $c>0$ is chosen arbitrarily, it follows that for every strictly positive function $f$, the function $w$ defined by $w:=c\tmR_a f>0$ is a viscosity subsolution. \\
On the other hand, let $(\tmL,D(\tmL))$ be the infinitesimal generator of the standard Brownian motion. Let $c\geq1$ and set $w=c\tmR_a f\in D(\tmL)$. Let us show that $w$ is a classical supersolution associated with $(\tmL,D(\tmL))$ to
\begin{align*}
\min(aw-\tmL w-f,w)= 0,
\end{align*}
Indeed, we have
\begin{align*}
\min(aw-\tmL w-f,w)=\min(cf-f,\tmR_a f)\geq 0.
\end{align*}
The equality follows by \eqref{eqn:resolventproperties2} and the inequality follows since $c\geq 1$. Hence by \lemref{lem:classical is viscosity},  $w=c\tmR_a f\in D(\tmL)$ is a viscosity supersolution associated with $(\tmL,D(\tmL))$. Thus, it is also a viscosity supersolution associated with  $(\frac{1}{2}D_{xx},\cc^\infty(\R))$.\\
Therefore, for $c\geq 1$ the function $w=c\tmR_a f$ is a viscosity solution associated with $(\frac{1}{2}D_{xx},\cc^\infty(\R))$. Since $c\in [1,\infty)$ is arbitrarily chosen, the uniqueness is not satisfied. \qed

	\end{example}
\begin{remark}
It is worth mentioning that by \thmref{thm:maintheorem}, the viscosity solution associated with $(\frac{1}{2}D_{xx},D(\tmG^*))$ (where $D(\tmG^*):=\{v\in \C_*(\R);v-\ti{v}(\partial)\in \cc^\infty(\R)\}$) is unique (see \corref{cor:sdmex1}).
	\end{remark}

\subsection{Proof of the Main results}

\subsubsection{Proof of \propref{prop:g*vs}}\label{sec: proofmain1}

Before proving the main results, we need some preliminary results. 
 We start with the following lemma that gives the relation between the infinitesimal generator of the Feller semigroup $\{\tmP_t\}_{t\geq 0}$ and that of its extension $\{\mP_t\}_{t\geq 0}$.

\begin{lemma}\label{lem:coreoperator}
Let $\{\tmP_t\}_{t\geq 0}$ be a Feller semigroup on $\cz(\sE)$, whose infinitesimal generator is $(\tmL,D(\tmL))$ with a core $(\tmG,D(\tmG))$. Given  the Feller semigroup $\{\mP_t\}_{t\geq 0}$ defined by \eqref{eqn:extfs}, its infinitesimal generator $(\mL,D(\mL))$ satisfies
\begin{align}
\label{eqn:mLdefine}
\mL w&= \reallywidetilde{\tmL ( (w-w(\partial))|_\sE)} \,\mbox{  for each }w \in D(\mL),
\end{align}
with
\begin{align}\label{eqn:mLdefine1}
D(\mL)&=\{w\in \C(\sO); (w-w(\partial))|_\sE\in D(\tmL)\}.
\end{align}
Furthermore, suppose $(\mG,D(\mG))$ is the restriction of $(\mL,D(\mL))$ on $D(\mG)$ with

 \begin{align}\label{eqn:tgdefine0}
 \begin{split}
 D(\mG) &:= \{w\in \C(\sO); (w-w(\partial))|_{\sE} \in D(\tmG)\}\\
 \mG w &:=\reallywidetilde{\tmG  ((w-w(\partial))|_\sE)} \text{ for } w \in D(\mG)
 \end{split}
 \end{align}

 Then $(\mG,D(\mG))$ is also the core of the Feller semigroup $\{\mP_t\}_{t\geq 0}$.
\end{lemma}

\begin{proof}
	See Appendix~B.
\end{proof}

 Since $\{\mP_t\}_{t\geq 0}$ defined by \eqnref{eqn:extfs} is a conservative Feller semigroup, we  know \cite[Theorem I.9.4]{blumenthal2007markov} that there exists a corresponding Feller process $\ti{X}$ whose transition semigroup is $\{\mP_t\}_{t\geq 0}$ with the compact state space $\sO$. $\ti{X}$ is also a standard Markov process. Define the value function $\ti{V}$ of $\ti{X}$ by
\begin{align}\label{eqn:valuefunctiongh}
\ti{V}(x) := \sup_{\tau}\ti{\bfE}^x\Big[\int_{0}^{\tau}{e^{-as}\ti{f}(\ti{X}_s)\, \mathrm{d}s}+e^{-a\tau}\ti{g}(\ti{X}(\tau))\Big] \mbox{ for }x\in \sO.
\end{align}
One can check that all the conditions in \assref{ass:fscompact} are fulfilled. In fact, $\sO$ is compact; using \lemref{lem:coreoperator}, $(\mG,D(\mG))$ defined by \eqnref{eqn:tgdefine0} is the core of the Feller process $\ti{X}$ and $f,g\in \C_*(\sE)$ implies $\ti{f},\ti{g}\in \C(\sO)$. Then, by \thmref{thm:gvs=V}, the above value function $\ti{V}\in \C(\sO)$ is the unique viscosity solution associated with $(\mG,D(\mG))$ to
\begin{align}\label{eqn:vsg}
\min(a \ti{w} -\mG \ti{w} -\ti{f} ,\ti{w}- \ti{g}) = 0.
\end{align}

\begin{lemma}\label{lem:compric*}
Suppose the assumptions in \propref{prop:g*vs} hold. Assume that $w\in USC(\sE)$ \textup{(}\resp{$LSC(\sE)$}\textup{)} is a viscosity subsolution \textup{(}\resp{supersolution}\textup{)} associated with $(\tmG^*,D(\tmG^*))$ to \eqref{eqn:vsaaa}. Define the extension $\bar{w}$ on $\sO$ by
\begin{align}
\bar{w} := \begin{cases}
w(x)&\mbox{ for }x\in \sE\\
\max(\frac{\ti{f}(\partial)}{a},\ti{g}(\partial)) &\mbox{ for }x=\partial.
\end{cases}
\end{align}
If $\bar{w}\in USC(\sO)$ \textup{(}\resp{$LSC(\sO)$}\textup{)}, then $\bar{w}$ is a viscosity subsolution \textup{(}\resp{supersolution}\textup{)} associated with $(\mG,D(\mG))$ to \eqref{eqn:vsg}.
\end{lemma}
\begin{proof}
 Let $w\in USC(\sE)$ be a viscosity subsolution associated with $(\tmG^*,D(\tmG^*))$ to \eqnref{eqn:vsg*}. We want to show that $\bar{w}\in USC(\sO)$ is also a viscosity subsolution associated with $(\mG,D(\mG))$ to \eqnref{eqn:vsg}. Let $\phi\in D(\mG)$ such that $\phi-\bar{w}$ has a global minimum at $x$ in $\sO$ with $\phi(x) = \bar{w}(x)$, we want to show that
\begin{eqnarray}\label{eqn:pthmg*vs=V2}
\min{(a \phi(x)-\mG \phi(x)-\ti{f}(x),\phi(x)-\ti{g}(x))} \leq 0.
\end{eqnarray}
We distinguish two cases:\\
\textbf{(a)} Assume that $x= \partial$ (an absorbing point). Then, $\mG \phi(\partial) = 0$ for all $\phi \in D(\mG)$. In addition, since $\bar{w}(\partial) \leq \max(\ti{f}(\partial)/a,\ti{g}(\partial))$, \eqnref{eqn:pthmg*vs=V2} is satisfied.\\
\textbf{(b)} Assume that $x\in\sE$. Define $\phi^*\in C_*(\sE)$ by $\phi^*: = \phi|_\sE$. Since $\phi\in D(\mG)$, it follows from \eqnref{eqn:tgdefine0} that $\phi^*-\phi(\partial)\in D(\tmG)$. In addition, we claim that $\phi^*\in D(\tmG^*)$.

To see this, we first assume that $\sE$ is not compact. Then $\phi(\partial) = \ti{\phi^*}(\partial)$ and thus  $\phi^*-\ti{\phi^*}(\partial)=(\phi-\phi(\partial))|_\sE\in D(\tmG) $ (since $\phi \in D(\ti{\tmG}))$. Hence $\phi^*\in D(\tmG^*)$. Next, we assume that $\sE$ is compact. In this case, $\phi \in D(\mG)$ means $\phi\in C(\sE_{\partial})$ and $(\phi-\phi(\partial))|_{\sE}\in D(\tmG)$, that is, $\phi|_{\sE}-\phi(\partial)\in D(\tmG)$. 
Using the fact that $1\in D(\tmG)$, we obtain $\phi^*:=\phi|_{\sE}\in D(\tmG)$. Therefore, since $\ti{\phi^*}(\partial)=0$ by the compactness of $\sE$ and $\phi^*\in D(\tmG)$, it follows from \eqnref{eqn:g*define} that $\phi^*\in D(\tmG^*)$. The claim is thus proved.

Next, recall that $\phi-\bar{w}$ has a global minimum at $x$ in $\sO$ with $\phi(x) = \bar{w}(x)$. Hence, using $\phi^* := \phi|_\sE$ and $w =\bar{w}|_\sE$, it follows that $\phi^*-w$ has a global minimum at $x$ in $\sE$ with $\phi^*(x) = w(x)$. Combining this with the fact that $\phi^*\in D(\tmG^*)$, and since $w$ is viscosity subsolution associated with $(\tmG^*,D(\tmG^*))$ to \eqnref{eqn:vsg*}, we have
\begin{eqnarray}\label{eqn:pthmg*vs=V1}
\min{(a \phi^*(x)-\tmG^* \phi^*(x)-f(x),\phi^*(x)-g(x))} \leq 0.
\end{eqnarray}

Since $f=\ti{f}|_\sE, g=\ti{g}|_\sE$ and  $\phi^* = \phi|_\sE$, in order to prove that \eqnref{eqn:pthmg*vs=V2} holds when $x\in \sE$, it is enough to show that
\begin{eqnarray}\label{eqn:pthmg*vs=V111}
\mG\phi(x) \geq \tmG^*\phi^*(x).
\end{eqnarray}
It follows from \eqnref{eqn:tgdefine0}  (\resp{\eqnref{eqn:g*define}}) that $\mG \ti{\phi^*}(x) = \tmG (\ti{\phi^*}- \ti{\phi^*}(\partial))|_\sE (x)$ (respectively, $\tmG^*\phi^*(x) = \tmG(\phi^*-\ti{\phi^*}(\partial))(x)$) and thus $\mG\ti{\phi^*}(x) = \tmG^* \phi^*(x)$. That is, \eqref{eqn:pthmg*vs=V111} becomes \begin{eqnarray}\label{eqn:pthmg*vs=V112}
\mG\phi(x)\geq \mG\ti{\phi^*}(x).
\end{eqnarray}
 Two cases are distinguished. \\
\textbf{(i)} Assume that $\sE$ is not compact. By the uniqueness of the extension, we have $\phi=\ti{\phi^*}$ and $\mG\phi(x)=\mG\ti{\phi^*}$. \\
\textbf{(ii)} Assume that $\sE$ is compact. By the definition of $\ti{\phi^*}$ (see\eqref{ext1}), we have $\ti{\phi^*}(y)=\phi(y)$ for any $y\in \sE$ and $\ti{\phi^*}(\partial)=0$. In addition, since $\phi-\bar{w}$ has a global minimum at $x$ in $\sO$ and $w\in USC(\sE)$, we have $\phi(y)\geq \bar{w}(y)$ for any $y\in \sE_{\partial}$ and thus $\phi(\partial)\geq \bar{w}(\partial)=\max(\ti{f}(\partial)/a,\ti{g}(\partial))=0=\ti{\phi^*}(\partial)$, since $\sE$ is compact. This indicates that $\ti{\phi^*}-\phi$ has a positive maximum equal $0$ at $x$ in $\sO$. Since $(\mG, D(\mG))$ is the core of $(\mL, D(\mL))$ (see\lemref{lem:coreoperator}), it follows from \thmref{thm:hyr} that, $(\mG, D(\mG))$ satisfies the positive maximum principle and thus $\mG\phi(x)\geq \mG\ti{\phi^*}(x)$.

The viscosity supersolution can be proved in a similar way.
\end{proof}

\begin{proof}[Proof of \propref{prop:g*vs}]
\textbf{(1)} We will prove that $V:=\ti{V}|_{\sE}$ is a viscosity solution in $\C_*(\sE)$ associated with $(\tmG^*,D(\tmG^*))$ to \eqnref{eqn:vsg*}. We first prove that $V$  is a viscosity subsolution associated with  $(\tmG^*,D(\tmG^*))$ to \eqnref{eqn:vsg*}. Let $x\in \sE$ and $\psi \in D(\tmG^*)$  such that $\psi-v_e$ has a global minimum at $x$ with $\psi(x)=V(x)$. There are two cases:\\
\textbf{(i)} Suppose that $\sE$ is compact. Then $\ti{\psi}(\partial)-V(\partial)=0$.\\
\textbf{(ii)} Suppose that $\sE$ in not compact. Since $\sE$ is a dense open subset of $\sO$, we have $\ti{\psi}(\partial)-\ti{V}(\partial)\geq 0$.\\
It follows that $\ti{\psi}-\ti{V}$ has a global minimum at $x$ in $\sO$ with $\ti{\psi}(x) = \ti{V}(x)$. Moreover, since $\psi\in D(\tmG^*)$, then by the definition \eqref{eqn:g*define} of $ D(\tmG^*)$, we have $\psi-\ti{\psi}(\partial)\in D(\tmG)$. Hence, $(\ti{\psi}-\ti{\psi}(\partial))|_{\sE} \in D(\tmG)$. Therefore, $\ti{\psi} \in D(\mG)$ by the definition \eqref{eqn:tgdefine0} of $D(\mG)$. Since the value function $\ti{V}$ defined by \eqnref{eqn:valuefunctiongh} is a viscosity subsolution associated with $(\mG,D(\mG))$ to \eqnref{eqn:vsg}, we have
$$
\min{(a\ti{\psi}(x)-\mG \ti{\psi}(x)-\ti{f}(x),\ti{\psi}(x)-\ti{g}(x))} \leq 0.
$$
Furthermore, since $x\in \sE$, using \eqref{eqn:tgdefine0}, we have $\mG \ti{\psi}(x) = \tmG (\ti{\psi}- \ti{\psi}(\partial))|_\sE (x)$, and using \eqnref{eqn:g*define}, we have $\tmG^*\psi(x) = \tmG(\psi-\ti{\psi}(\partial))(x)$. Hence, $\mG \ti{\psi}(x) = \tmG^*\psi(x)$. Therefore,
$$\min{(a \psi(x)-\tmG^* \psi(x)-f(x),\psi(x)-g(x))} \leq 0.$$

We can also prove in a similar way that $V$ is a viscosity supersolution. The existence is then proved, that is,  $V = \ti{V}|_\sE$  is a viscosity solution associated with  $(\tmG^*,D(\tmG^*))$ to \eqnref{eqn:vsg*}.

\textbf{(2)} Next, we show that $V|_\sE$ is the unique viscosity solution associated with $(\tmG^*,D(\tmG^*))$. The idea here is to prove that if $w\in\C_*(\sE)$ is a viscosity solution associated with $(\tmG^*,D(\tmG^*))$ to \eqnref{eqn:vsg*}, then $\ti{w}\in\C(\sO)$ is a viscosity solution associated with $(\mG,D(\mG))$ to \eqnref{eqn:vsg}. Hence, the result will follow since the viscosity solution associated with $(\mG,D(\mG))$ to \eqref{eqn:vsg} is unique. Using \lemref{lem:compric*}, if $w\in \C_*(\sE)$ is a viscosity solution associated with  $(\tmG^*,D(\tmG^*))$ to \eqnref{eqn:vsg*}, its extension $\bar{w}$ is the unique viscosity solution associated with  $(\mG,D(\mG))$ to \eqnref{eqn:vsg} which is the value function $V$ defined by \eqnref{eqn:valuefunctiongh}. This completes the proof of the uniqueness and the proposition.
\end{proof}

\subsubsection{Proof of \thmref{thm:maintheorem}}\label{sec: proofmain2}


For $\sE$ compact, since $\C_b(\sE) = \cz(\sE)$, the existence and uniqueness follow  \thmref{thm:gvs=V}. Thus, we only need to consider the case $\sE$ not compact.

\textbf{Existence:}
Using \thmref{thm:vsexist}, the viscosity solution associated with $(\tmG^*,D(\tmG^*))$ to \eqref{eqn:vsaaa} is the value function provided that $(\tmG^*,D(\tmG^*))$ is an $a$-generator.

Let us show that $(\tmG^*,D(\tmG^*))$ is an $a$-generator. By Dynkin's formula and the argument preceding \corref{cor:lviscositysolutionexists}, we have $(\tmG,D(\tmG))$ is an $a$-generator. Let us consider the restriction of $(\tmG^*,D(\tmG^*))$ to the space of constant functions. Since $\sE$ is not compact, using \eqref{eqn:g*define}, we have $\tmG^* 1= \tmG 0=0$. Hence $\{S_1(t)\}_{t\geq 0}$ given by \eqref{eqn:defasup} (with $w=1$) is an $(\mF_t,P^x)$ uniformly integrable martingale for $a>0$ and thus $(\tmG^*,D(\tmG^*))$ is an $a$-generator.

\textbf{Uniqueness:}
For the uniqueness, let $\{\phi_n\}_{n\in \mathbb{N}}$ be an increasing sequence in $\cz(\sE)$ converging pointwisely to the constant function $1$. By Dini's theorem, $\{\phi_n\}_{n\in N}$ converges to $1$ locally uniformly. Let $C\geq \max(\ninf{f},\ninf{g})$. Define $f_n^- := \phi_n\cdot (f+C)-C$ and $g_n^- := \phi_n\cdot (g+C)-C$. Then $\{f_n^-\}_{n\in \mathbb{N}}$ and $\{g_n^-\}_{n\in \mathbb{N}}$ are in $\C_*(\sE)$ and increasing.

Let $w\in \C_b(\sE)$ be a viscosity solution  associated with $(\tmG^*,D(\tmG^*))$ to \eqref{eqn:vsaaa}, which satisfies $w \geq -C$, and define
\begin{align*}
    w_n^-(x) := \sup_\tau \bfE^x\big[\int_0^\tau e^{-as}f_n^-(X(s))\diff s+e^{-a\tau}g_n^-(X(\tau))\big].
\end{align*}
By the existence proof, $w_n^-$ is a viscosity solution to
\begin{align}\label{eqn:proofwg*fn-}
\min(aw_n^--\tmG^* w_n^- - f_n^-, w_n^--g_n^-)=0.
\end{align}
Since $f\geq f_n^-$ and $g\geq g_n^-$ in $\sE$, $w$ is a viscosity supersolution to \eqref{eqn:proofwg*fn-}.
Therefore, by \lemref{lem:compric*}, $w\geq w_n^-$ for all $n\in \mathbb{N}$.

Similarly, let $f_n^+ := \phi_n\cdot (f-C)+C$ and $g_n^+ := \phi_n\cdot (g-C)+C$ and $w_n^+$ is
\begin{align*}
    w_n^+(x) := \sup_\tau \bfE^x\big[\int_0^\tau e^{-as}f_n^+(X(s))\diff s+e^{-a\tau}g_n^+(X(\tau))\big].
\end{align*}
It is the viscosity solution to
\begin{align}\label{eqn:proofwg*fn+}
\min(aw_n^+-\tmG^* w_n^+ - f_n^+, w_n^+-g_n^+)=0.
\end{align}
Then, similarly, since $w\leq C$ and $w$ is the viscosity subsolution to \eqref{eqn:proofwg*fn+}, by \lemref{lem:compric*}, then $w\leq w_n^+$.

Therefore, since $w_n^+\geq w\geq w_n^-$ for all $n\in\mathbb{N}$, to prove the uniqueness, it is enough to show that $\lim_{n\rightarrow}w_n^+(x) = \lim_{n\rightarrow}w_n^-(x) = V(x)$. We have the following inequalities,
\begin{align*}
V(x) - w_n^-(x) \leq& \sup_\tau \bfE^x\bigg[\int_0^\tau e^{-as}(1-\phi_n(X(s)))(f(X(s))-\ninf{f})\diff s\\
&+e^{-a\tau}(1-\phi_n(X(\tau)))](g(X(\tau))-\ninf{g})\bigg]\\
\leq & \sup_\tau \bfE^x\bigg[\int_0^\tau e^{-as}(1-\phi_n(X(s)))\ninf{f}\diff s+e^{-a\tau}(1-\phi_n(X(\tau)))\ninf{g}\bigg]\\
\leq & C\Big\{\tmR_a(1-\phi_n)(x)+\sup_{\tau}\bfE^x\big[e^{-a\tau}(1-\phi_n(X(\tau)))\big]\Big\}.
\end{align*}
By \cite[Theorem 3.2]{schilling1998conservativeness}, we know that $\tmR_a(1-\phi_n)$ converges to $0$ locally uniformly. Then, we only need to prove that $\{u_n\}_{n\geq 0}$, with $u_n(x):= \sup_{\tau}\bfE^x\big[e^{-a\tau}(1-\phi_n(X(\tau)))\big]$ converges to $0$ locally uniformly. As shown in \cite[Proposition 2.1]{palczewski2010finite}, for any compact set $K\subseteq \sE$, $T>0$ and $\varepsilon>0$, there exists a compact set $L_\varepsilon\subseteq \sE$ such that
\begin{align}
\sup_{x\in K}\bfP^x(X(s)\not\in L_\varepsilon \mbox{ for some }t\in [0,T])<\varepsilon.
\end{align}
Therefore for any $\mF_t$-stopping time $\tau$, for all $x\in K$, we have 
\begin{align*}
\bfE^x\big[e^{-a\tau}(1-\phi_n(X(\tau)))\big]=&\bfE^x\big[e^{-a\tau}(1-\phi_n(X(\tau)))\ind_{\tau<T}+e^{-a\tau}(1-\phi_n(X(\tau)))\ind_{\tau>T}\big]\\
\leq& \bfE^x\big[e^{-a\tau}(1-\phi_n(X(\tau)))\ind_{\tau<T}\big]+e^{-aT}\\
\leq& \bfE^x\big[e^{-a\tau}(1-\phi_n(X(\tau)))\ind_{\tau<T}\ind_{\{X(s)\not\in L_\varepsilon \mbox{ for some }t\in [0,T]\}}\big]\\
&+\bfE^x\big[e^{-a\tau}(1-\phi_n(X(\tau)))\ind_{\tau<T}\ind_{\{X(s)\in L_\varepsilon \mbox{ for all }t\in [0,T]\}}\big]+e^{-aT}\\
\leq& \varepsilon+\sup_{x\in L_\varepsilon}(1-\phi_n(x))+e^{-aT}.
\end{align*}
Since $L_\varepsilon$ is compact and $\{\phi_n\}_{n\in\mathbb{N}}$ converges to $1$ locally uniformly, $\sup_{x\in L_\varepsilon}(1-\phi_n(x))$ converges to $0$ as $n\rightarrow\infty$. Since $\varepsilon$, $K$ and $T$ are all arbitrarily chosen, $u_n$ converges to $0$ locally uniformly. Therefore, $\{w_n^-\}_{n\in \mathbb{N}}$ converges to $V$ locally uniformly. Similarly, we have $\{w_n^+\}_{n\in \mathbb{N}}$ converges to $V$ locally uniformly. This completes the proof of the uniqueness. \hfill \ensuremath{\Box}

\section{Structure of the optimal stopping value functions}\label{structsec}

In this section, we related the viscosity solution to some existing results, using martingale approach. First, we introduce some preliminary lemmas.

\begin{lemma}\label{lem:matmet}
Given $u\in LSC(\sE)$ \textup{(}\resp{$u\in USC(\sE)$}\textup{)}, define the process $\{M(t)\}_{t\geq 0}$ by
\begin{align}\label{eqn:matdef}
    M(t) := e^{-at} u(X(t))+\int_0^t e^{-as}f(X(s))\diff s.
\end{align}
Suppose there exists an open subset $\mO \subseteq \sE$ such that $\{M(t\wedge\tau_{\mO})\}_{t\geq 0}$ is a $(\mF_t,\bfP^{x})$ uniformly integrable supermartingale \textup{(}\resp{submartingale}\textup{)} for all $x\in \mO$. Then, the following claims hold.
\begin{enumerate}
    \item For all $\phi\in D(\tmG^*)$ such that $\phi-u$ has a global maximum \textup{(}\resp{minimum}) at $x_0\in \mO$ with $\phi(x_0)=u(x_0)$,
    \begin{align}
        a\phi(x_0)-\tmG^*\phi(x_0)-f(x_0)\leq (\geq)0.
    \end{align}
    \item Additionally, suppose there exists a subset $K_0\subseteq \sE$ such that $X$ satisfies $\bfP^{x_0}\big[X(\tau_\mO)\in K_0\big]=1$ for some $x_0\in \mO$. Then for all $\psi\in D(\tmG^*)$ such that $\psi-w$ has a maximum \textup{(}\resp{minimum}\textup{)} in $K_0$ at $x_0$ with $\psi(x_0)=w(x_0)$, we have
    \begin{align}\label{eqn:stsupm}
        a\phi(x_0)-\tmG^*\psi(x_0)-f(x_0)\leq (\geq)0.
    \end{align}
\end{enumerate}
\end{lemma}
\begin{proof} The proof is similar as the proof of \thmref{thm:vsexist}. Here, we only prove the statement (2) since the statement (1) follows when $K_0 = \sE$ in the statement (2). Let $\psi\in D(\tmG^*)$ such that $\psi-u$ has a maximum (\resp{minimum}) in $K_0$ at $x_0$ with $\psi(x_0)=u(x_0)$.
Let the process  $\{S(t)\}_{t\geq 0}$ be defined by
\begin{align}\label{eqn:proofstsupm}
    S(t):=e^{-at}\psi(X(t))+\int_0^t e^{-as}(a\psi(X(s))-\tmG^*\psi(X(s)))\diff s.
\end{align}
By Dynkin formula, since $\psi\in D(\tmG^*)$, $\{S(t)\}_{t\geq 0}$
is a $(\mF_t,\bfP^{x_0})$ uniformly integrable martingale. We first assume $x_0$ is a point not absorbing. Let $\delta>0$ and $\tau_\delta$ defined by \eqnref{eqn:taudelta}. Since $\{M(t\wedge\tau_{\mO})\}_{t\geq 0}$ is a $(\mF_t,\bfP^{x_0})$ uniformly integrable supermartingale, we have
\begin{align*}
u(x_0)&\geq \bfE^{x_0}\bigg[\int_0^{\tau_\delta}e^{-as}f(X(s))\diff s+e^{-a \tau_\delta}u(X(\tau_\delta))\bigg] \\
&\geq \bfE^{x_0}\bigg[\int_0^{\tau_\delta}e^{-as}f(X(s))\diff s+e^{-a \tau_\delta}\psi(X(\tau_\delta))\bigg]\\
&\geq \bfE^{x_0}\bigg[\int_0^{\tau_\delta}e^{-as}(f(X(s))+\tmG^*\psi(X(s))-a\psi(X(s))) \diff s\bigg]+\psi(x),
\end{align*}
where the last inequality follows from the optional stopping theorem. Thus \eqref{eqn:stsupm} is proved in an analogous way as \eqref{eqvnat1} in  \thmref{thm:vsexist}. The case the non-absorbing point $x_0$ can be proved the same way as in \thmref{thm:vsexist}.
\end{proof}

\begin{corollary}\label{cor:sup>value}
    Let $u\in LSC(\sE)$ be bounded from the below. Suppose that its corresponding process $\{M(t)\}_{t\geq 0}$ defined by \eqref{eqn:matdef} is a supermartingale. If $u\geq g$, then $u$ is a viscosity supersolution to
    \begin{align}
        \min(aw-\tmG^*w-f,w-g)=0
    \end{align}
    and $u\geq V$, where $V$ is the value function defined by \eqref{eqn:valuefunction}.
\end{corollary}
\begin{proof}
Since $\{M(t)\}_{t\geq 0}$ is a supermaringale, by \lemref{lem:matmet}, $u$ is a viscosity supersolution associated with $(\tmG^*,D(\tmG^*))$ to $$aw-\tmG^*w-f = 0.$$ Since $u\geq g$, $u$ is also a viscosity supersolution to
$$\min(aw-\tmG^*w-f,w-g)=0.$$
By the comparison principle (see \thmref{thm:maintheoremcompariosncb}), we have $u\geq V$.
\end{proof}

\begin{corollary}\label{cor:sub<value}
    Let $u\in USC(\sE)$ be bounded from above. Suppose there exists an open subset $\mO\in \sE$ such that its corresponding process $\{M(t\wedge \tau_\mO)\}_{t\geq 0}$ defined by \eqref{eqn:matdef} is a submartingale.
    \begin{enumerate}
        \item If $u(x)\leq g(x)$ for all $x\not\in \mO$, then $u(y)\leq V(y)$ for all $y\in \mO$, where $V$ is the value function defined by \eqref{eqn:valuefunction}.
        \item Additionally, suppose there exists a subset $\bar{\mO}\subseteq K_0\subseteq \sE$ such that $X$ satisfies $\bfP^{x}\big[X(\tau_\mO)\in K_0\big]=1$ for all $x\in \mO$. If $u(x)\leq g(x)$ for all $x\in K_0\setminus \mO$, then $u(y)\leq V(y)$ for all $y\in \mO$.
    \end{enumerate}
\end{corollary}
\begin{proof}
    As in \lemref{lem:matmet}, we give the proof of (2) and (1) by setting $K_0=\sE$. First define the function $u_-$ by
    \begin{align*}
        u_-(x):=\begin{cases}
        u(x)&\mbox{ for }x\in \mO\\
        u(x)\wedge g(x)&\mbox{ for }x\not\in\mO.
        \end{cases}
    \end{align*}
    Since $\mO$ is an open subset with $\bar{\mO}\subseteq K_0$, $g$ is a continuous function and $u(x)\leq g(x)$ for $x\in K_0\setminus \mO$, we have $u_-\in USC(\sE)$.

    Similarly with \corref{cor:sup>value}, by \lemref{lem:matmet}, $u_-$ is a viscosity subsolution to
    $$aw(x)-\tmG^*w(x)-f(x)=0\mbox{ for }x\in \mO.$$
    Since $u_-(x)\leq g(x)$ for all $x\not\in \mO$, then $u$ is a viscosity subsolution to
    $$\min(aw-\tmG^*w-f,w-g)=0.$$

    By the comparison principle (see \thmref{thm:maintheoremcompariosncb}), we have $u_-\leq V$ and then $u(x)\leq V(x)$ for all $x\in \mO$.

\end{proof}

Combing \corref{cor:sup>value} and \corref{cor:sub<value}, the following result suggests that the value function characterized in the proof of \cite[Theorem 3.1]{alili2005some} coincides with the viscosity solution to \eqref{eqn:vsg*}.
\begin{theorem}\label{thm:martingaleapproach}
Let $u\in \C_b(\sE)$ and $\mO\subseteq \sE$ be an open subset. Suppose there exists a subset $K_0$ such that $\bar{\mO}\subseteq K_0\subseteq \sE$ and $X$ satisfies $\bfP^{x}\big[X(\tau_\mO)\in K_0\big]=1$ for all $x\in \mO$. Additionally, suppose the following hold.
\begin{enumerate}
    \item $\{M(t\wedge\tau_\mO)\}_{t\geq 0}$ is a uniformly integrable martingale,
    \item $\{M(t)\}_{t\geq 0}$ is a supermartingale,
    \item $u\geq g$ and $u(x)= g(x)$ for $x\in K_0\setminus\mO$,
\end{enumerate}
Then, $u(x)=V(x)$ for all $x\in \mO$.
\end{theorem}

The above theorem gives a classical method to find the optimal stopping value function using the martingale characterization. It is traditionally used for one-dimensional process to find explicit solution for optimal stopping for diffusion. We should mention that the martingale approach usually does not require the continuity and boundedness of the reward functions $f$ and $g$. (See for example \cite{beibel2001optimal}.)

\section{Applications}\label{sec:levyprocess}

\subsection{Viscosity properties of value functions for optimal stopping problems}
In this section, we apply the results to study viscosity properties for optimal stopping problems for some processes satisfying \assref{ass:fscompact} and whose core fulfils the conditions of our main theorems. Let us mention that many traditional processes studied in the literature satisfy those assumptions. We revisit the optimal stopping using viscosity approach developed in the paper. To our knowledge, optimal stopping problems for  Brownian motion jumping at boundary and semi-Markov process have not been studied in the literature using viscosity approach. Recall that the objective function is given by
\begin{align}\label{eqn:valuefunction*}
    V(x):= \sup_{\tau}\bfE^x\Big[\int_0^\tau e^{-as}f(X(s))\diff s+e^{-a\tau}g(X(\tau))\Big].
\end{align}
Let $\sE$ be a space to be determined in each example. In this section, we always assume that $$a>0 \text{ and } f,g\in \C_b(\sE).$$ We will first use \thmref{thm:vsexist} to show that the value function given by \eqnref{eqn:valuefunction*} is a viscosity solution. Let us start with L\'{e}vy processes on the state space $\sE = \R^n$.

\subsubsection{L\'{e}vy Processes} \label{exa:levy} Here, we assume that $X=\{X(t)\}_{t\geq 0}$ is a L\'evy process on $\sE = \R^n$. It is known (see for example   \cite[Theorem 17.10]{kallenberg2006foundations}) that $X=\{X(t)\}_{t\geq 0}$ is a Feller process.
Its core $(\tmG_{\text{L\'evy}}, D(\tmG_{\text{L\'evy}}))$ is given by
\begin{align}\label{eqn:levygenerator}
\tmG_{\text{L\'evy}} w(x) =& \ell\cdot \triangledown w(x)+\frac{1}{2} div\,Q\,\triangledown w(x)\nonumber\\
&+\int_{\R^n\setminus \{0\}} \left(w(x+y)-w(x)-\triangledown w(x)\cdot y \ind_{|y|<1}\right)\nu (\diff y),
\end{align}
for $x\in \R^n$ and $w\in D(\tmG_{\text{L\'evy}}) := \C_0^\infty (\R^n)$, where $\ell\in \R^n$ is a vector, $Q \in \R^{n\times n}$ is a symmetric positive semi-definite matrix, $\nu$ is a positive Radon measure satisfying $\int_{\R^n\setminus \{0\}}\min{(|y|^2,1)\nu (\diff y)}<\infty$ and $\cz^\infty(\R^n)$ denotes the space of all infinitely differentiable functions and itself and all its derivatives belong to $\cz(\R^n)$.
We have the following result from \thmref{thm:maintheorem}.

\begin{proposition}\label{prop:levyvsexp}
	Assume that $X=\{X(t)\}_{t\geq 0}$ is a L\'evy process whose core  $(\tmG_{\text{L\'evy}}, D(\tmG_{\text{L\'evy}}))$ is described above. Then the value function $V$ given by \eqref{eqn:valuefunction*} is the unique viscosity solution $w\in \C_b(\R^n)$ associated with $(\tmG^*_{\text{L\'evy}}, D(\tmG^*_{\text{L\'evy}}))$ to
	\begin{equation}\label{eqn:levyvs}
	\min(a w - \tmG^*_{\text{L\'evy}}w-f, w-g)=0,
	\end{equation}
where $D(\tmG^*_{\text{L\'evy}})=\{v\in \C_*(\R^n); v-\ti{v}(\partial)\in\cz^\infty(\R^n) \}$.
\end{proposition}

\begin{remark}
    Similar optimal stopping problem was studied in  \cite{alili2005some,mordecki2002optimal}. In particular, the authors look at perpetual put options for one dimensional L\'{e}vy process with $f=0$ and $g(x)=K-e^{\beta x}$, where $K>0$ and $\beta>0$. More precisely, the value function has the following form
    \begin{align}
        V(x) := \sup_{\tau}\bfE^x\big[e^{-a\tau}(K-e^{\beta X(\tau)})^+\big].
    \end{align}
    Let us note that \cite{alili2005some} used a martingale approach similar to \thmref{thm:martingaleapproach} to prove that the value function is solution to a martingale problem.  Alternatively, we can use \propref{prop:levyvsexp} to show that the value function is the unique viscosity solution to the associated HJB equation.
\end{remark}

Let us now assume that the process $X=\{B(t)\}_{t\geq 0}$ is a one dimensional \textsl{standard Brownian motion}, that is, a Feller process with state space $\sE=\R$ and core $(\tmG_{BM},D(\tmG_{BM}))$ given by
\begin{align}\label{eqn:sbmgenerator}
\begin{split}
D(\tmG_{BM})&:=\{u\in \C_0(\R)\cap \C^2(\R);\, D_xu, D_{xx}u\in \cz(\R)\},\\
\tmG_{BM}u(x)&:=\frac{1}{2}D_{xx} u(x) \text{ for }x\in \R.
\end{split}
\end{align}
\thmref{thm:vsexist} gives us the freedom to choose larger domains than $D_0(\tmG^{BM})$, for example,
\begin{align}
D(\tmG^*_{BM}):=&\{u\in \C_*(\R)\cap\C^2 (\R);D_xu,D_{xx} u\in \C_0(\R)\}, \label{eqn:sbmgenerator11}\\
D(\tmG^{(b)}_{BM}):=&\{u\in\C_b(\R)\cap\C^2(\R);D_xu,D_{xx}u\in \C_b(\R)\}, \label{eqn:sbmgenerator12}\\
D(\tmG^{(p)}_{BM}):=&\{u\in \C^2 (\R); D_{xx}u\in \C_b(\R)\mbox{ and there exists } K>0 \nonumber\\
&\mbox{ such that }|u(x)|^2\leq K(1+|x|^{2})\mbox{ for all }x\in \R \}\label{eqn:sbmgenerator13}.
\end{align}
Using \thmref{thm:vsexist} and \thmref{thm:maintheorem}, we have the following result:
\begin{corollary}\label{cor:sdmex1}
	Assume that $X=\{B(t)\}_{t\geq 0}$ is a one dimensional \textsl{standard Brownian motion}. Then the value function $V$ given by \eqref{eqn:valuefunction*} is the unique viscosity solution $w\in\cz(\R)$ associated with $(\tmG_{BM}, D(\tmG^*_{BM}))$ (\resp{$(\tmG_{BM},D(\tmG^{(b)}_{BM}))$, $(\tmG_{BM}, D(\tmG^{(p)}_{BM}))$}) to
	\begin{equation}
	\min(a w - \tmG_{BM}w-f, w-g)=0.
	\end{equation}
\end{corollary}
\begin{proof}
Let us first observe that $(\tmG_{BM}, D(\tmG^*_{BM}))$ corresponds to $(\tmG^*, D(\tmG^*))$ and $(\tmG_{BM}, D(\tmG_{BM}))$ corresponds to $(\tmG, D(\tmG))$ in \thmref{thm:maintheorem}. 
Let $w\in (\tmG_{BM}, D(\tmG^*_{BM}))$ (\resp{$(\tmG_{BM}, D(\tmG^{(b)}_{BM})), (\tmG_{BM}, D(\tmG_{BM}^{(p)}))$}).
	Using It\^{o}'s formula, the  process $\{S_w(t)\}_{t\geq 0}$ given by
	$$S_w(t) := w(X_0)-e^{-at}w(X(t))-\int_0^t e^{-as}\Big(aw(X(s))-\frac{1}{2}D_{xx} w(X(s))\Big)\diff s \mbox{ for }t\geq 0$$
	is a $(\mF_t,\bfP^x)$-uniformly integrable martingale for $a> 0$ and $x\in \R$. Using \defref{def:agen} the operator $(\tmG_{BM}, D(\tmG^*_{BM}))$ (\resp{$(\tmG_{BM}, D(\tmG^{(b)}_{BM}))$, $(\tmG^{BM}, D(\tmG^{(p)}_{BM}))$}) is an $a$-generator. Hence by \thmref{thm:vsexist}, the value function $V$ defined by \eqnref{eqn:valuefunction*} is a viscosity solution associated with $(\tmG^{BM}, D(\tmG^*_{BM}))$ (\resp{$(\tmG_{BM}, D(\tmG^{(b)}_{BM}))$, $(\tmG_{BM}, D(\tmG^{(p)}_{BM}))$}). The uniqueness follows from \thmref{thm:maintheorem} since $(\tmG_{BM}, D(\tmG^*_{BM}))$ (\resp{$(\tmG_{BM}, D(\tmG^{(b)}_{BM}))$, $(\tmG_{BM}, D(\tmG^{(p)}_{BM}))$}) corresponds to $(\tmG^*,D(\tmG^*))$.
\end{proof}

In the next section we consider examples of one dimensional diffusion processes on the positive half line $\sE=[0,\infty)$ that behave like a standard Brownian motion with different boundary behaviours at boundary $0$.

\subsubsection{Diffusion on $\sE=[0,\infty)$}\label{exa:onedimdiff}
Let $X=\{X(t)\}_{t\geq 0}$ be a diffusion process on $[0,\infty)$. Then the generator $(\tmG_{BC},D(\tmG_{BC}))$ of $X=\{X(t)\}_{t\geq 0}$ is given by
\begin{align}
\begin{split}
D(\tmG_{BC})&:=\{u\in \C_0([0,\infty))\cap\C^2([0,\infty));\,  D_xu,D_{xx}u\in \cz([0,\infty))\},\\
\tmG_{BC}u(x)&:=\frac{1}{2}D_{xx} u(x) \text{ for } x\in [0,\infty).
\end{split}
\end{align}
The operator $(\tmG_{BC},D(\tmG_{BC}))$ does not satisfy the positive maximum principle at $0$ unless we add some appropriate conditions at boundary $0$. Let us consider the following processes with appropriate domain
\begin{enumerate}
	\item \textsl{Reflected Brownian motion}: $D(\tmG_{ref}):=\{u\in D(\tmG_{BC});\,D_xu(0)=0\}$;
	\item \textsl{Sticking Brownian motion}: $D(\tmG_{stk}):=\{u\in D(\tmG_{BC});\, D_{xx}u(0)=0\}$;
	\item \textsl{Sticky reflecting Brownian motion}: $D(\tmG_{stkref}):=\{u\in D(\tmG_{BC});\, D_{xx}u(0)=c D_xu(0)\}$, where $c\in (0,\infty)$.
	\item \textsl{Brownian motion with jump at the boundary}: $D(\tmG_{jump}):=\{u\in D(\tmG_{BC});\, D_{xx}u(0)=\lambda \int_{[0,\infty)}(u(0)-u(y)) \mu (\diff y)\}$, where $\lambda>0$ and $\mu$ is a probability measure.
\end{enumerate}
We have the following result from \thmref{thm:vsexist} and \thmref{thm:maintheorem}:
\begin{proposition}
Assume that $X=\{X(t)\}_{t\geq 0}$ is a reflected Brownian motion (\resp{sticking Brownian motion, sticky reflecting Brownian motion}). Then the value function $V$ given by \eqref{eqn:valuefunction*} is a unique viscosity solution in $\C_b(\sE)$ associated with $(\tmG_{BC},D(\tmG_{ref}^*))$ (\resp{$(\tmG_{BC},D(\tmG_{stk}^*))$, $(\tmG_{BC},D(\tmG_{stkref}^*)))$, $(\tmG_{BC},D(\tmG_{jump}^*)))$.}
\end{proposition}

\begin{proof}
It follows from the fact that the above processes are Feller processes. 
\end{proof}

Now, consider the reflected Brownian motion and define 
\begin{align*}
D(\tmG_{ref}^+):=& \{u\in \C_b([0,\infty))\cap\C^2([0,\infty)); D_xu,D_{xx}u\in \C_b([0,\infty)) \mbox{ and } D_xu(0)\geq 0\},\\
D(\tmG_{ref}^-):=& \{u\in \C_b([0,\infty))\cap\C^2([0,\infty)); D_xu,D_{xx}u\in \C_b([0,\infty)) \mbox{ and } D_x u(0)\leq 0\}.
\end{align*}
\begin{corollary}
Assume that $X=\{X(t)\}_{t\geq 0}$ is a reflected Brownian motion. Then the value function $V$ given by \eqref{eqn:valuefunction} is the unique function in $\cz(\R^+)$ which is both a viscosity supersolution associated with $(\tmG_{BC},D(\tmG^+_{ref}))$ and a viscosity subsolution associated with  $(\tmG_{BC},D(\tmG^-_{ref}))$.
\end{corollary}

\begin{proof}
Since $w\in D(\tmG^+_{ref})$ (\resp{$D(\tmG^-_{ref})$}), the process $\{S_w(t)\}_{t\geq 0}$ given by
$$S_w(t)=w(X_0)-e^{-at}w(X(t))-\int_0^t e^{-as}\left(aw(X(s))-\tmG_{BC}w(X(s))\right)\diff s \mbox{ for }t\geq 0$$
is a $(\mF_t,\bfP^x)$ uniformly integrable supermartingale  (\resp{submartingale}). Hence, \thmref{thm:vsexist} suggests that the value function defined by \eqnref{eqn:valuefunction} is a viscosity supersolution (\resp{subsolution}) associated with $(\tmG_{BC},D(\tmG^+_{ref}))$ (\resp{$(\tmG_{BC},D(\tmG^-_{ref}))$}).
As for the uniqueness, we only need to show that it holds for the operator $(\tmG_{BC},D(\tmG^*_{ref}))$, where
$$
D(\tmG^*_{ref})=\{u\in \C_*([0,\infty))\cap\C^2([0,\infty));\,  D_xu, D_{xx}u\in \cz(\R^+) \text{ and } Du(0)=0\}.
$$
This follows from \thmref{thm:maintheorem}. Therefore, it leads to the desired result, since $(\tmG_{BC},D(\tmG^*_{ref}))$ can be seen as the restriction of $(\tmG_{BC},D(\tmG^+_{ref})\cap D(\tmG^-_{ref}))$ on $D(\tmG^*_{ref})$.
\end{proof}

\begin{remark}
In the above example, we consider the simplest cases of standard Brownian motion with the state space $[0,\infty)$. More generally, Feller \cite{feller1952parabolic,feller1954diffusion, feller1957generalized} constructs
Markov processes up to a specific regular boundary point $0$ with the boundary condition given by
	$$c_1 w(0)-c_2 Dw(0)+c_3 D_{xx}w(0)=0,$$
	for $c_1,c_2,c_3\geq 0$ with $c_1+c_2+c_3=1$. However, we have not considered the cases in which $c_1>0$, for example, the ``Dirichlet condition'' $w(0)=0$, or the ``Robin condition" $c_1 w(0)-c_2 Dw(0)=0$. The reason is that when $c_1>0$ the above Markov processes may be killed upon reaching $0$ and thus, does not coincide  with \defref{def:fellp} of Feller process.  Nevertheless, our method is still applicable as demonstrated in the following example.
\end{remark}

Let $(\tmG_{kill}, D(\tmG_{kill}))$ be an operator defined by
\begin{align}
\begin{split}
D(\tmG_{kill})&:=\{u\in \cz((0,\infty))\cap\C^2((0,\infty));\,D_xu,D_{xx}u\in \cz((0,\infty))\},\\
\tmG_{kill} u(x) &:=  \frac{1}{2}D_{xx}u(x) \text{ for }x\in (0,\infty).
\end{split}
\end{align}

Using Proposition \ref{prop:g*vs}, we have the corollary below:
 \begin{corollary}\label{cor:visckil1}
 	Suppose $f,g\in \cz((0,\infty))$ and $a>0$. Then there exists a unique viscosity solution  $w\in \cz((0,\infty))$ associated with $(\tmG_{kill},D(\tmG^{*}_{kill}))$ to
 	$$\min(a w-\tmG^{*}_{kill} w-f,w-g)=0,$$
 	where $D(\tmG^*_{kill}) = \{u\in \C_*((0,\infty))\cap\C^2((0,\infty));\,D_{xx}u\in \cz((0,\infty))\}$.
 	\end{corollary}
 \begin{proof}
  It is known (see for example \cite{feller1954diffusion}) that $(\tmG_{kill}, D(\tmG_{kill}))$ is the core of a Feller semigroup. Hence, the result follows from \propref{prop:g*vs}.
\end{proof}

\begin{remark}
Assume that $X=\{X(t)\}_{t\geq 0}$ is a standard Brownian motion. We show in \cite{Suhang2016b} that under additional assumptions, the unique viscosity solution given in \corref{cor:visckil1} is the value function to the following optimal stopping problem:
$$ V(x) = \sup_{\tau}\bfE^x\bigg[\int_0^{\tau\wedge \tau_0} e^{-as}f(X(s))\diff s+e^{-a\tau}g(X(\tau))\ind_{\tau<\tau_0}\bigg] \mbox{ for } x\in (0,\infty),$$
where $\tau_0 :=\inf\{t> 0; X(t)\not\in (0,\infty)\}$.
\end{remark}

%
%

 In the next section, we wish to establish viscosity properties of the value function of the optimal stopping problem \eqref{eqn:valuefunction*}, when $X$ is a diffusion with piecewise coefficients. Such problem with discontinuous function $f$ and $g=0$ was studied in \cite{belomestny2010optimal,ruschendorf2008class} using a ``modified" free boundary approach. The definition of viscosity solution given in \cite[Definiton 4.2 and 4.3]{belomestny2010optimal} does not ensure that the value function is the unique solution. In this paper, assuming that $f,g\in\C_b(\sE)$ and using different definition of viscosity solution, we show the viscosity property of the value function.

\subsubsection{Diffusion with piecewise coefficients}
 We start by constructing a diffusion process $X=\{X(t)\}_{t\geq 0}$ with piecewise coefficients. Let $\sigma$, $\rho$ and $\mu$ be three bounded real valued measurable functions.  Suppose $\sigma|_{\R\setminus J}\in \C^1_b(\R\setminus J)$ and $\mu|_{\R\setminus J},\rho|_{\R\setminus J}\in \C_b(\R\setminus J)$, where $J$ is a set in $\R$ without cluster points and contains all the discontinuous points of the functions $\sigma$, $\mu$ and $\rho$. In addition, suppose there exists $\lambda>0$ such that $\sigma,\mu>\lambda$.   We know from \cite[Propositions~2.1, 2.2 and 2.6]{lejay2015one} that there exists a Feller process $X$ with continuous paths whose infinitesimal generator is given by
\begin{align}\label{eqdpc1}
\begin{split}
D(\tmG_{pw})&:= \{w\in \cz(\R);\, D_xu, D_{xx} u \text{ exists in }\R\setminus J, \tmG_{pw}u\in \cz(\R)\notag\\
&\qquad\text{ and }\sigma(x^-)D_xu(x^-)=\sigma(x^+)D_xu(x^+) \text{ for all } x\in J\}.\\
\tmG_{pw} u(x) &:=\left\{
\begin{array}{ll}
\frac{\rho(x)}{2} D_x(\sigma(x) D_x w)x(x)+\mu x()D_xu(x_x )&\hbox{ for } x\in \R\setminus J,\\
\frac{\rho(x)}{2} D_x(\sigma D_x u)_x((x)x^+)+\mu D_xu(x^+) & \hbox{ for }x\in J,
\end{array}
\right.
\end{split}
\end{align}
As a consequence of \thmref{thm:maintheorem}, we have the following result.
\begin{corollary}\label{cor:skew1}
	Let $X=\{X(t)\}_{t\geq 0}$ be Feller process whose core $(\tmG_{pw},D(\tmG_{pw} ))$ is given by \eqref{eqdpc1}. Then the value function $V$ given by \eqref{eqn:valuefunction*} is the unique viscosity solution $w\in \C_b(\R)$ associated with $(\tmG_{pw}^*,D(\tmG_{pw}^* ))$, where
	\begin{align}
	D(\tmG^*_{pw}):=& \{u\in \C_*(\R);\, D_xu,D_{xx} u \text{ exists in }\R\setminus J, \tmG_{pw}u\in \cz(\R)\notag\\
	&\text{ and }\sigma(x^-)D_xu(x^-)=\sigma(x^+)D_xu(x^+) \text{ for all } x\in J\}.
	\end{align}
\end{corollary}

In particular, \cite[Chapter~VII, Exercise~1.23]{revuz2013continuous} provides an example of \textsl{Skew Brownian motion} with parameter $\beta\in (0,1)$. 
Heuristically speaking, it is constructed by a Brownian motion reflected at zero which enters the positive half line with probability $\frac{\beta+1}{2}$ (\resp{the negative half with probability $\frac{1-\beta}{2}$}) when it reaches zero. Its core is given by
\begin{align}\label{eqsbm1}
\begin{split}
D(\tmG_{skew}) &:= \{u\in \cz(\R);\,\, D_xu, D_{xx}u \mbox{ exists in }\R\setminus \{0\}\mbox{ and converges to }0 \mbox{ at infinity},\\
	&\qquad D_{xx} w(0^-)=D_{xx} w(0^+)\mbox{  and }
	\beta D_x w(0^+) =(1-\beta)D_x w(0^-)\},\\
\tmG_{skew} u(x) &:=\left\{
\begin{array}{ll}
\frac{1}{2}D_{xx}u(x)&\hbox{ for  } x\in \R\setminus\{0\},\\
\frac{1}{2}D_{xx}u (0^+)& \hbox{ for }x=0.
\end{array}
\right.
\end{split}
\end{align}
Again, \thmref{thm:maintheorem} yields the following result
\begin{corollary}\label{cor:skew2}
	Let $X=\{X(t)\}_{t\geq 0}$ be a skew Brownian motion with parameter $\beta\in (0,1)$. Then the value function \eqref{eqn:valuefunction} of the stopping problem \eqref{eqn:payofffunction}-\eqref{eqn:valuefunction}  is the unique viscosity solution $w\in \C_b(\R)$ associated with $( \tmG^*_{skew},D(\tmG^*_{skew}))$, where
	\begin{align}\label{eqsbm2}
	D(\tmG^*_{skew}) :=& \{u\in \C_*(\R);\,\, D_xu, D_{xx}u \mbox{ exists in }\R\setminus \{0\}\mbox{ and converges to }0 \mbox{ at infinity},\notag\\
	&\, D_{xx} u(0^-)=D_{xx} u(0^+)\mbox{  and }
	\beta D_xu(0^+) =(1-\beta)D_xu(0^-)\}.
	\end{align}
	
\end{corollary}

\begin{remark}
Observe that $D(\tmG^*_{skew})$ in the above example does not contain any smooth function unless its derivative is $0$. Therefore, showing that a function has the viscosity property at $0$ means test functions $\phi$ as described in \defref{def:vs} are continuous but are not smooth at $0$.  This leads to additional technical difficulty  in the proof of the uniqueness when using the traditional method. This is due to the fact that this method is based on smoothness of test function and properties of elliptic or parabolic differential equations.
\end{remark}

\subsection{Perturbation} Perturbation is a powerful method to transform a known Feller process to a new Feller process. We first introduce the following lemma which enables to construct the Feller semigroup using perturbation.
\begin{lemma}\cite{bottcher2013levy}\label{lem:jcf}
Let $(\tmG,D(\tmG))$ be the infinitesimal generator of some Feller semigroup on $\cz(\sE)$. Assume that $B:\cz(\sE)\rightarrow\cz(\sE)$ and $B$ is bounded, that is, there exists $C>0$ such that  $\sup_{u\in \cz(\sE)}\frac{\ninf{Bu}}{\ninf{u}}\leq C$. Additionally suppose $(B,\cz(\sE))$ satisfies the positive maximum principle. Then, $(\tmL+B,D(\tmL))$ is the infinitesimal generator of some Feller semigroup on $\cz(\sE)$.
\end{lemma}

Using this method, we provide constructions of Feller processes via perturbation. The first example is the Feller process with large jumps.

\subsubsection{Compound Poisson Operator} Let $\{X(t)\}_{t\geq 0}$ be a Feller process with the state space state space $[0,\infty)$ and core given by $(\tmG,D(\tmG))$. Define a bounded operator $B$ by
    \begin{align}
        Bu(x):= \lambda\int_0^\infty(u(x)-u(x-y))\diff \mu(y),
    \end{align}
    where $\mu$ is a probability distribution function defined on $(0,\infty)$ and $\lambda$ is the intensity parameter. Then by \lemref{lem:jcf}, $(\tmG+B,D(\tmG^*))$ is the infinitesimal generator of some Feller process $\{Y(t)\}_{t\geq 0}$. For example, let $\{B(t)\}_{t\geq 0}$ be a standard Brownian motion and $\mu(x)=1-e^{-\gamma x}$ be the distribution function of an exponential random variable with parameter $\gamma$. Let $\{X_b(t)\}_{t\geq 0}$ be a compound Poisson process with the intensity $\lambda >0$ and the jump height following an exponential distribution with parameter $\gamma$. Then, in this case, one can choose $\{Y(t),\mF_t^Y\}_{t\geq 0}$ as
    \begin{align}
        Y(t) = Y(0)+B(t)+X_b(t)\mbox{ for }t\geq 0
    \end{align}
    with core $(\tmG_{ref}+B,D(\tmG_{ref}))$, where $\mF_t^Y$ is the natural filtration of $\{Y(t)\}_{t\geq 0}$. Thus $\{Y(t)\}$ is still a Feller process. Hence viscosity solution approach can be used to characterise the value function of the optimal stopping problem of $\{Y(t)\}_{t\geq 0}$. 







\subsubsection{Semi-Markov Process}

Let $\{T_i\}_{i\in \mathbb{N}}$ be a sequence of independent and identical (i.i.d.) random variables with cumulative density distribution function $P$. $\{T_i\}_{i\in \mathbb{N}}$ can be seen as the interarrival time of some random event. Additionally, let $\{Y_i\}_{i\in \mathbb{N}}$ be a sequence of i.i.d random variables defined on $\R$ with distribution function $F$. Let $S_n:=\sum_{i=1}^{n} T_i$ for $n=0,1,\ldots$ and the renewal process $N(t):=\max\{n;S_n\leq t\}$. Let $\{X(t)\}_{t\geq 0}$ be
    \begin{align}\label{eqn:semimark}
        X(t):= x+\sum_{i=1}^{N(t)}Y_i \text{ for }t\geq 0,
    \end{align}
    where $x$ is the initial state.
    For example when the interarrival time is the exponential distribution, $\{X(t)\}_{t\geq 0}$ is a compound Poisson distribution which is a Markov process. However, if the interarrival time does not follow the exponential distribution, $\{X_t\}_{t\geq 0}$ is not a Markov process but a semi-Markov process. We want to analyze the optimal stopping problem of
    \begin{align}
            V_{semi}(x) := \sup_{\tau}\bfE^x\big[\int_0^\tau e^{-as}f(X(s))\diff s+e^{-a\tau}g(X(\tau))\big],
        \end{align}
    where $a>0$ and $f,g\in \C_b(\R)$.
    \begin{remark}
        Optimal stopping problems of semi-Markov process has been studied in \cite{boshuizen1993general,muciek2002optimal}. The work \cite{boshuizen1993general} provides several applications of semi-Markov processes in real life, for example, job search and shock model ( see \cite[Section 1]{boshuizen1993general}.) In this section, we want to solve optimal stopping problems using viscosity approach and not the iterative approach as in \cite{boshuizen1993general,muciek2002optimal}.
    \end{remark}

    Assume that $P$ is an absolutely continuous function and $p$ is its continuous density function on $[0,\infty)$. Define $Q(x):=p(x)/(1-P(x))$ for $x\in [0,\infty)$. In addition, assume that $\lim_{x\rightarrow \infty}\frac{p(x)}{1-P(x)}=C$. Then, $Q$ has a continuous extension $\bar{Q}$ on $[0,\infty]$. Examples are:
    \begin{enumerate}
        \item \textsl{Mixture exponential distribution}: $P(x):=\sum_{i=1}^{m}w_i(1-e^{-\lambda_i x})$, where $\sum_{i=1}^{m}w_i=1$, $w_i>0$, $\lambda_i>0$ and $m$ is some positive integral. Its density function is given by $p(x)=\sum_{i=1}^{m}w_i \lambda_i e^{-\lambda_i x}$. Thus, $\lim_{x\rightarrow\infty}Q(x)=\lim_{x\rightarrow\infty}\frac{p(x)}{1-P(x)}=\min_{i=1,2,\ldots,m}\lambda_i$.
        \item Generalized beta prime distribution: Here, $P(x):= \frac{x}{1+x}$ and $p(x):=\frac{1}{(1+x)^2}$. Thus, $Q(x)=\frac{1}{1+x}$ for $[0,\infty)$.
    \end{enumerate}

    Let $\{\xi(t)\}_{t\geq 0}$ be the time from the last jumps of $\{X(t)\}_{t\geq 0}$ (for example if  $S_n$ is the time of the last jump at time $t$, $\xi(t)=t-S_n$). Then, the two dimensional process  $\{\xi(t),X(t)\}_{t\geq 0}$ is a Markov process (see for example \cite[Lemma 2, p290]{gihman1975theory}). Its  infinitesimal generator is defined by
    \begin{align}\label{eqn:gensemimark}
        \begin{split}
            D(\tmG) &:= \{u\in\cz([0,\infty]\times \R); u(\cdot,x)\in \C^1([0,\infty]) \mbox{ for all }x\in \R \mbox{ and }Du(\infty)=0\},\\
            \tmG u(s,x) &:= D_su(s,x)+Q(s)\int_\R(u(0,x+y)-u(s,x))\diff F(y) \mbox{ for }s\in [0,\infty]\mbox{ and }x\in \R.
        \end{split}
    \end{align}

    \begin{proposition}
        Assume that $X$ is a semi-Markov process defined by \eqref{eqn:semimark}.
        \begin{enumerate}
            \item There exists a unique viscosity solution $w\in \C_b([0,\infty]\times\R)$ associated with $(\tmG^*,D(\tmG^*))$ defined by \eqref{eqn:gensemimark} to
                \begin{align}
                    \min(a w-\tmG^* w-\bar{f},w-\bar{g})=0,
                \end{align}
            where $\bar{f}(s,x) = f(x)$ and $\bar{g}(s,x)=g(x)$ for all $s\in [0,\infty]$ and $x\in \R$.
            \item The value function can be characterized by $V(x)=w(0,x)$.
            \item Let $\{\xi(t)\}_{t\geq 0}$ be the time from the last jump. Let $\gamma(x) := \inf\{s\in[0,\infty] ; w(s,x)=g(x)\}$. Then the optimal stopping time is
            \begin{align}
                \tau^*:= \inf\{t\in [0,\infty); \xi(t) = \gamma(X(t))\}.
            \end{align}
        \end{enumerate}
    \end{proposition}

    \begin{proof}
        First, we prove that \eqref{eqn:gensemimark} is an infinitesimal generator of some Feller semigroup. Since $(D_s,D(\tmG))$ is the generator of some Feller semigroup, by \lemref{lem:jcf}, we only need to prove: (i)  $B$ is defines from $\cz(\sE)$ to $\cz(\sE)$, (ii) $B$ is bounded and (iii) $B$ satisfies the positive maximum principle, where
        \begin{align}
        \begin{split}
            Bu(s,x) &:= Q(s)\int_\R(u(0,x+y)-u(s,x))\diff F(y)\mbox{ for }s\in [0,\infty]\mbox{ and }x\in \R.
        \end{split}
        \end{align}
        Let $u\in \cz([0,\infty]\times \R)$. We have
        \begin{enumerate}
            \item[(i)] $\int_R u_0(x+y)\diff y\in \cz(\R)$, where $u_0(x):=u(0,x)$ for $x\in \R$. Then $\int_\R(u(0,x+y)-u(s,x))\diff F(y)=\int_\R u(0,x+y)\diff y-u(s,x)$ implies that $B:\cz([0,\infty]\times \R)\rightarrow\cz([0,\infty]\times \R)$ follows since $Q\in \C_b([0,\infty])$ and $Q\geq 0$.
            \item[(ii)] $|\int_\R u(0,x+y)\diff y-u(s,x)|\leq 2\ninf{u}$ and $Q$ is bounded. Then $B$ is bounded.
            \item[(iii)] If $(s_0,x_0)$ is the global maximum point and $u(s_0,x_0)\geq0$, $Bu(s_0,x_0)=Q(s_0)\int_\R(u(0,x+y)-u(s,x))\diff y\leq 0$.
        \end{enumerate}
        Therefore, $(\tmG,D(\tmG))$ is a Feller generator. Furthermore, define
        \begin{align}
            W(\xi,x):=\sup_{\tau}\bfE^{\xi,x}\big[\int_0^\tau e^{-as}\bar{f}(\xi(s),X(s))\diff s+e^{-a\tau}\bar{g}(\xi(t),X(\tau))\big].
        \end{align}
        Since the semi-Markov process $\{X(t)\}_{t\geq 0}$ and the Markov process $\{\xi(t),X(t)\}_{t\geq 0}$ have the same filtration and probability measure, we have $W(0,x)=V(x)$ for $x\in \R$. Since $(\tmG,D(\tmG))$ is the generator of the Feller process $\{\xi(t),X(t)\}_{t\geq 0}$, we can use \thmref{thm:maintheorem} to show (1) and (2) and \thmref{thm:obtainost}  to show (3).
    \end{proof}

\begin{remark}
In this example, we have not derived an explicit value function for the optimal stopping problem. However, in \cite{Dai20172}, we suggest an iterative scheme to find the value function. 
\end{remark}


\section{Explicit solutions}\label{secexpsolu}
In this section, we apply the results obtained in Section 7 to explicitly derive the solution to the following optimal stopping problem: Find $\tau^*$ such that
\begin{align}\label{eqn:osvfref}
V(x) := \bfE^x[e^{-a\tau^*}g(X(\tau^*))]=\sup_{\tau}\bfE^x\big[e^{-a\tau}g(X(\tau))\big]  \text{ for } x\in [0,\infty),
\end{align}
where $g(x)=(c_2-x)^+-(c_1-x)^+$ with $c_1<c_2\in \mathbb{R}$ and $\{X\}_{t\geq 0}$ is a process to be described. $g(x)$ can be understood as the straddle option which is the difference of two options.
\subsection{Reflected Brownian Motion}\label{refsecappli22} In this section, let $c_1,c_2\in \mathbb{R}$ with $c_1<c_2$ and suppose $\{X(t)\}_{t\geq 0}$ is a reflected Brownian motion reflected at $0$ with state space $E = [0,\infty)$ with core
\begin{align}
\begin{split}
D(\tmG_{ref}) :=& \{u\in C^2_0([0,\infty)); D_xu(0)=0\},\\
\tmG_{ref}u(x):=& \frac{1}{2}D_{xx} u(x)\text{ for }x\in [0,\infty).
\end{split}
\end{align}
Our aim is to find the explicit optimal stopping time of problem \eqref{eqn:osvfref} based on \thmref{thm:maintheorem}. The following corollary is a direct consequence.
\begin{proposition}\label{cor:valrefvs}
	The value function $V$ given by \eqref{eqn:osvfref} is the unique viscosity solution $w\in \C_b([0,\infty))$ associated with $(\tmG_{ref}^*,D(G_{ref}^*))$ to
	\begin{align}\label{eqn:refvsc1c2}
	    \min(aw-\frac{1}{2}D_{xx} w,w-g)=0,
	\end{align}
	where $g(x)=(c_1-x)^+-(c_2-x)^+ \text{ for } x\in [0,\infty).$
\end{proposition}

\begin{proof}
	This result directly follows from \thmref{thm:maintheorem} by setting $f=0$ and $g(x)=(c_2-x)^+-(c_1-x)^+$ for $x\in [0,\infty)$.
\end{proof}

In order to find $\tau^*$, we first need to compute $V$ explicitly as shown below.

\begin{corollary}\label{cor:exvfrefc1c2} Let $X$ be a reflected Brownian motion reflected at $0$.
 Let $C$ define by
		\begin{align}\label{eqn:finddref}
		C := \min\{p>0; p(e^{\sqrt{2a}}+e^{\sqrt{-2a}})\geq g(x)\},
		\end{align}
		where $a$ is the discount rate. Then, the value function $V=w$, where
\begin{align}\label{eqn:vfrefc1c2}
w(x):=
\begin{cases}
C(e^{\sqrt{2a}x}+e^{-\sqrt{2a}x}) &\text{ for } x\in [c_1,x^*),\\
g(x) &\text{ for } x\in [x^*,\infty),
\end{cases}
\end{align}
and
	\begin{align}\label{eqn:refx*}
		x^* = \min\{x;C(e^{\sqrt{2a}x}+e^{-\sqrt{2a}x})=g(x)\}.
	\end{align}
Additionally, the optimal stopping time is $\tau^*=\left\{t\geq 0;X(t)\in [x^*,c_2]\right\}.$
\end{corollary}
\begin{proof}
	Let us show that $w$ defined by \eqref{eqn:vfrefc1c2} is a viscosity solution. By definition of $C$ in \eqref{eqn:finddref}, $w(x)\geq g(x)$ for $x\in [0,x^*)$. Using \eqref{eqn:vfrefc1c2}, we get $w\geq g$. In what follows, we show the viscosity property for different values of $x$.
	
\noindent \textbf{Case 1.} Assume that $x\in[0,x^*)$. It is clear from \eqref{eqn:vfrefc1c2} that $w$ is twice differentiable at $x$ and we have
		\begin{align}\label{eqn:solverefc1c2}
		aw(x)-\frac{1}{2}D_{xx} w(x) = 0 \text{ for }x\in [0,x^*).
		\end{align}
		Since $w(x) \geq g(x)$ for $x\in [0,\infty)$, we have
		$$\min(aw(x)-\frac{1}{2}D_{xx} w(x),w(x)-g(x))=0.$$
		Let $\phi\in D(G^*_{ref})$ such that $\phi-w$ has a maximum (\resp{minimum}) at $x$ with $\phi(x)-w(x)=0$. We first show that  $D_{xx}\phi(x)\leq (\geq) D_{xx} w(x)$.  Assume that $x\in (0,x^*)$. Then $x$ is an interior point. Since $w$ is twice differentiable at $x$ and $\phi-w$ has a maximum (\resp{minimum}) at $x$,  we have $D_{xx}\phi(x)\leq (\geq) D_{xx} w(x)$. Assume now that $x=0$. Since $\phi\in D(\tmG^*_{ref})$, we have $D\phi(0)=0$. Using $Dw(0)=0$, we have $D(\phi-w)(0)=0$. Furthermore, since $\phi-w$ has a maximum (\resp{minimum}) at $x=0$, it follows that $D_{xx}(\phi-w)(0)\leq (\geq) 0$. Therefore,
		\begin{align*}
		&\min(a\phi(0)-\frac{1}{2}D_{xx} \phi(0),\phi(0)-g(0))\\
		\geq(\leq)& \min(a\phi(0)-\frac{1}{2}D_{xx} w(0),\phi(0)-g(0))\\
		=&0.
		\end{align*}
		Hence, $w$ satisfies viscosity property at $x$.
		
\noindent \textbf{Case 2.} Assume that $x=x^*$ Since $w(x^*)=g(x^*)$, the viscosity subsolution property is satisfied. Then, we only need to show the viscosity supersolution property. \\
Let $\phi\in D(\tmG^*_{ref})$ such that $\phi-w$ has a maximum at $x^*$ with $\phi(x^*)-w(x^*)=0$.
Define $w_0(x):= C(e^{\sqrt{2a}x}+e^{-\sqrt{2a}x})$. By \eqref{eqn:finddref} and \eqref{eqn:vfrefc1c2}, we have $w_0(x)\geq w(x)$ for all $x\in [0,\infty)$ and $\phi(x^*)=w(x^*)=w_0(x^*)$.
It implies that $\phi-w_0$ also has a maximum at $x^*$ with $\phi(x^*)-w_0(x^*)=\phi(x^*)-w(x^*)=0$. Hence, since $\phi-w_0$ is twice differentiable and $x^*$ is interior point, $D_{xx}(\phi-w_0)\leq 0$. Therefore,
		\begin{align*}
		& \min\left(a\phi(x^*)-\frac{1}{2}D_{xx} \phi(x^*),\phi(x^*)-g(x^*)\right)\\
		\geq& \min\left(aw_0(x^*)-\frac{1}{2}D_{xx} w_0(x^*),0\right)\\
		= &0.
		\end{align*}
		Then, the viscosity supersolution property is satisfied.\\
\noindent \textbf{Case 3} Assume that $x>x^*$. Since $w(x)=g(x)$, we only need to show the viscosity supersolution. It can be proved similarly with Case 1. The result follows by uniqueness of the viscosity solution (\thmref{thm:maintheorem}.)\\
Moreover, the optimal stopping time can be obtained using \thmref{thm:obtainost}.
\end{proof}




Next, we consider a standard Brownian motions with jumps at the boundary $0$.

\subsection{Brownian motion with jump at boundary}

Let $\{X(t)\}_{t\geq 0}$ be a standard Brownian motion which has nonlocal behavior at $0$ and state space $E = [0,\infty)$. Then $\{X(t)\}_{t\geq 0}$ is a Feller process 
whose core is defined by (see for example \cite{taira2004semigroups})
\begin{align}\label{eqn:genjb}
    \begin{split}
    D(\tmG_{jump})&:= \{u\in \C^2_0([0,\infty)); D_{xx} u(0)=2\lambda\int_0^\infty (u(y)-u(0))\diff F(y)\},\\
    \tmG_{jump} u(x) &:= \frac{1}{2}D_{xx} u(x) \mbox{ for }x\in [0,\infty),
    \end{split}
\end{align}
where $\lambda$ is a positive constant and $F$ is a probability distribution function on $(0,\infty)$. The process stays at zero for a positive length of exponential waiting time with parameter $\lambda$ and then jump back to a random point in $(0,\infty)$ with a probability defined by the distribution function $F$. Let $V_{jump}$ be the value function of the optimal stopping problem \eqref{eqn:osvfref}. Then, we have the following result:

\begin{proposition}\label{explisoltBMjb}
    Suppose there exists a solution such that $u(x) = C_1 e^{-\sqrt{2a}x}+C_2 e^{\sqrt{2a}x}$ for $x\in [0,\infty)$, where $C_1,C_2\in \R$ and satisfy
    \begin{enumerate}
        \item $u\geq g$,
        \item There exists $x^*\in [0,\infty)$ such that $u(x^*)=g(x^*)$,
        \item $g$ is a viscosity supersolution to
        $$aw(x)-\frac{1}{2}D_{xx}w(x)= 0 \mbox{ for }x\in (x_*,\infty),$$
        \item $a(C_1+C_2)=\lambda\int_0^{x^*}u(y)\diff F(y)+\lambda\int_{x^*}^\infty g(y)\diff F(y)-\lambda u(0)$.
    \end{enumerate}
    Then,
    \begin{align}
        V_{jump}(x)=u_-(x):=\begin{cases}
        u(x) &\text{ for }x\in [0,x^*),\\
        g(x) &\text{ for }x\in [x^*,\infty).
        \end{cases}
    \end{align}
\end{proposition}
\begin{proof}
Using \thmref{thm:maintheorem}, we only need to show that $u_-$ is a viscosity solution associated to $(\tmG_{jump}^*,D(\tmG_{jump}^*)$. We only prove the viscosity supersolution property and subsolution property can be shown similarily. Since $u_-\geq g$, we simply need to show that for any $\phi\in D(\tmG^*_{jump})$ such that $\phi\leq u_-$ and $\phi(x_0) = u_-(x_0)$, we have simply
\begin{align}\label{eqn:phi<jb}
a\phi(x_0)-\frac{1}{2}D_{xx}\phi(x_0)\geq 0.
\end{align}
\textbf{Case 1.} Suppose $x_0\in (x^*,\infty)$. \eqref{eqn:phi<jb} follows by the condition (3).\\
\noindent \textbf{Case 2.} Suppose $x_0=x^*$. Since $\phi-u_-$ has a global maximum at $x^*$ by condition (1) and (2) and $\phi$ and $u_-$ are twice differentiable at $x_0$, we have $D_{xx}\phi(x_0)\leq D_{xx}u_-(x_0)$. \\
\noindent \textbf{Case 3.} Suppose $x_0\in (0,x^*)$. Since $u(x)=u_-(x)$ for all $x\in (0,x^*)$ and $au-\frac{1}{2} D_{xx}u=0$, \eqref{eqn:phi<jb} holds.\\
\noindent\textbf{Case 4.} Suppose $x_0=0$. By the definition of $D(\tmG^*_{jump})$, we have
\begin{align*}
    a\phi(0)-\frac{1}{2}D_{xx}\phi(0)&=a\phi(0)-\lambda\int_0^\infty \phi(x)\diff F(x)+\lambda \phi(0)\\
    &\leq au_-(0)-\lambda\int_0^\infty u_-(x)\diff F(x)+\lambda u_-(0)\\
    &=0,
\end{align*}
where the first inequality follows from $u_-\geq \phi$ and $u_-(0)=\phi(0)$ and the last equality from condition (4). Hence, \eqref{eqn:phi<jb} holds when $x_0=0$. Therefore, we conclude that $u_-$ is a viscosity supersolution. The case of the viscosity subsolutionccan be shown analogously.
\end{proof}
The following figure shows the evolution of the value function with fixed jump size at boundary.

	\begin{figure}[H]
		\caption{Value functions against initial state}\label{fig:vfc0}
		\centering
		\includegraphics[scale=0.6]{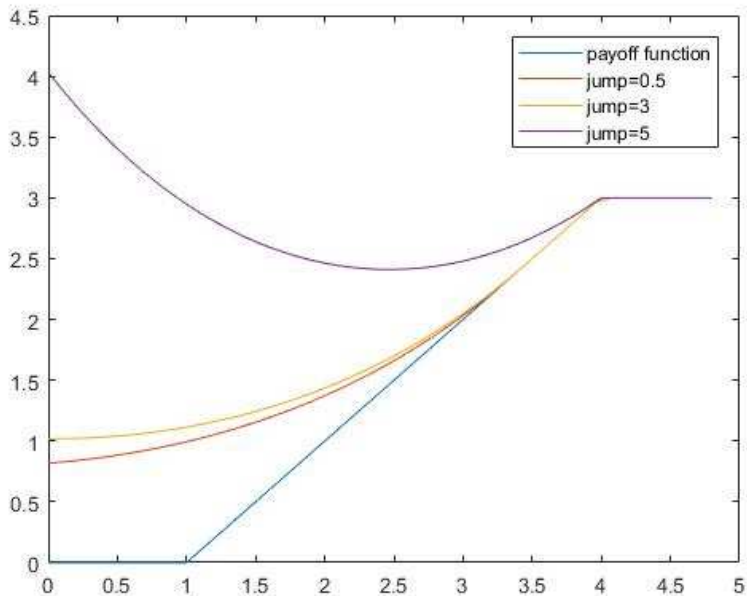}
	\end{figure}
In Figure~\ref{fig:vfc0}, we assume that the jump size is fixed at $0.5$ (respectively $3$ and $5$) and the parameter $\lambda=1$. The graph shows that the value function and exercise point increases with the jump size. We can also mention that the construction of the value function by the viscosity solution can generally be used under weaker condition as compared to the smooth fit principle. Since $g$ is not differential, the smooth fit principle may failed for example if the jump size is equal to $5$.

\subsection{Regime switching boundary}
In order to construct a regime switching boundary Feller diffusion, we first construct a regime switching Feller process. Let $\mathcal{S}:=\{1,2,\ldots,N\}$ be a finite discrete space, where $N$ is a positive integer. Let $(\mathcal{A}_i,D(\mathcal{A}_i))$ be the infinitesimal generators of some Feller semigroups on $\cz(\sE)$. Then,  define the operator $(\mathcal{A},D(\mathcal{A}))$ as follows:
\begin{align}
    \begin{split}
    D(\mathcal{A}_{regime}) &:= \{u\in\cz(S\times \sE); u(i,\cdot)\in D(\tmG_i)\},\\
    \mathcal{A}_{regime} u(i,x) &:= \mathcal{A}_i u_i(x)\mbox{ for }i\in S\mbox{ and }x\in \sE,
    \end{split}
\end{align}
where $u_i(x) := u(i,x)$.
By Hille-Yosida theorem, the above generator is the infinitesimal generator of some Feller semigroup. In addition, define the bounded operator
\begin{align}
    \mathcal{F}_{regime}u(i,x) := \sum_{j\in N}q_{ij}(x)(u(j,x)-u(i,x)),
\end{align}
where $q_{ij}\in \C_b(\sE)$ and $q_{ij}\geq 0,\,\, i,j\in \mathbb{N}$. Since $\mathcal{F}_{regime}$ satisfies the positive maximum principle and $\mathcal{F}_{regime}:\cz(\sE)\rightarrow\cz(\sE)$, the operator $((\mathcal{A}_{regime}+\mathcal{F}_{regime},D(\mathcal{A}_{regime})))$ is the infinitesimal generator of some Feller semigroup.

Next, we construct a regime switching boundary Feller diffusion, that is, the boundary condition is affected by a Markov chain $\{Z(t)\}_{t\geq 0}$ with the state space $\{1,2\}$. The intensity matrix of the chain is given by
\[\begin{bmatrix}
-q_1 & q_1\\ q_2 & -q_2
\end{bmatrix},\]
where $q_1,q_2>0$. Let $\{Z(t),X(t)\}_{t\geq 0}$ be a Feller process on the state space $\{1,2\}\times [0,\infty)$. $\{X(t)\}_{t\geq 0}$  is a one-side diffusion which  behaves like Brownian motion in $(0,\infty)$ but is modulated at $0$. More precisely, when $X(t)$ touches $0$, it either become a sticky Brownian motion or reflected Brownian motion. We denote by $Z(t)=1$ the state for sticky Brownian motion and $Z(t)=2$ the tate for reflected Brownian motion. Its infinitesimal generator $(\tmG,D(\tmG))$ is defined by:
\begin{align}
\begin{split}
    D(\tmG)&:=\{u\in \cz(\{1,2\}\times[0,\infty); u_i\in\C^2([0,\infty))\mbox{ for  }i=1,2,\\
    &\qquad D_{xx}u(1,0)=0 \mbox{ and }D_{x}u(2,0)=0\},\\
    \tmG u(i,x)&:=\frac{1}{2}D_{xx}u(i,x)+q_i u(3-i,x)-q_i u(i,x)\mbox{ for }i=1,2\mbox{ and }x\in [0,\infty),
\end{split}
\end{align}
where $u_i(x)=u(i,x)$ for all $(i,x)\in\{1,2\}\times [0,\infty)$. As a consequence of \thmref{thm:maintheorem}, we have the following characterisation of the value function for the optimal stopping problem \eqref{eqn:osvfref}:

\begin{corollary}There exists a unique pair of viscosity solution $V_1, V_2 \in \C_b([0,\infty))$ such that $V_1$ is a viscosity solution associated with $(\tmG_1,D(\tmG_1))$ to
$$\min((a+q_1)w-\tmG_1 w-q_1 V_2,w-g(1,\cdot))=0,$$
and $V_2$ is a viscosity solution associated with $(\tmG_2,D(\tmG_2))$ to
$$\min((a+q_2)w-\tmG_2 w-q_2 V_1,w-g(2,\cdot))=0.$$

Additionally, assume that $u(i^*,\cdot)$ is a viscosity supersolution to
$$ag(i^*,x)-\tmG_{i^*}g(i^*,x)+q_{i^*}(g(3-i^*,x)-g(i^*,x))=0, \mbox{ for }x\in (x^*,\infty).$$ Then, $V(i^*,x)=u(i^*,x)$ for $x\in [0,x^*)$ and $V(i^*,x)=g(i^*,x)$ for $x\in [x^*,\infty)$. Furthermore,

  $V(3-i_*,\cdot)$ is the unique viscosity solution associated with $(\tmG_{3-i^*},D(\tmG_{3-i^*}))$ to
    \begin{align*}
        \min((a+q_{3-i^*})w-\tmG^*_{3-i^*}w-q_{3-i^*}V(i^*,\cdot),w-g(3-i^*,\cdot))=0
    \end{align*}
 \end{corollary}

In order to derive explicit value function, we define fundamental solutions for optimal stopping problem.
Let
\begin{align}
    u_k(i,x) &= \alpha_{ik} e^{\beta_k x}\\
    v_j(i,x)&:=\begin{cases}
    q_i\tmR^{(j)}_{a+q_i} g_{3-i}(x)\text{ for }i= j\\
    g_{i}(x)\text{ for }i\neq j,
    \end{cases}\\
    w_{j1}(i,x)&:=\begin{cases}
    e^{\gamma_j x}\text{ for }i= j\\
    0\text{ for }i\neq j,
    \end{cases}\\
    w_{j2}(i,x)&:=\begin{cases}
    e^{-\gamma_j x}\text{ for }i= j\\
    0\text{ for }i \neq j,
    \end{cases}
\end{align}
where $i,j=1,2$, $k=1,2,3,4$ $\beta_1=\sqrt{2a}$, $\beta_2=-\sqrt{2a}$, $\beta_3=\sqrt{2(a+q_1+q_2}$ and $\beta_4 = -\sqrt{2(a+q_1+q_2)}$, $\alpha_{1k}=1$  and $\alpha_{2k} = \frac{q_1}{a+q_1-\frac{1}{2}z_j^2}$ and $\gamma_j=\sqrt{2(a+q_j)}$.

\begin{lemma}\label{lem:rsvpl}
The following hold for $j=1,2$:
\begin{enumerate}
    \item For any $A_j\in \R$, $u:=\sum_{j=1}^4 A_j u_j$ is a solution to
    \begin{align}
        au(i,x)-\tmG^* u(i,x)=0 \text{ for }(i,x)\in \{1,2\}\times(0,\infty)
    \end{align}
    \item For any $B_k\in \R$, $w:=\sum_{k=1}^2 B_k w_{jk}+v_j$ is a solution to
    \begin{align}
        aw(i,x)-\tmG^* w(i,x)=0 \text{ for }(i,x)\in \{j\}\times (0,\infty),
    \end{align}
\end{enumerate}
\end{lemma}
\begin{proof}
The result simply follows from direct computations given the parameters. 	
	\end{proof}
	The subsequent result can be seen as a verification theorem for the value function.
\begin{proposition}\label{prop:rsexpli}
Assume that there exist $0\leq x_1^*\leq x_2^*<\infty$, $A_j\in \R$, $B_k\in \R$ for $j=1,2,3,4$ and $k=1,2$ such that the function
        \begin{align}
            u(i,x):=\begin{cases}
            \sum_{j=1}^4 A_j u_j(i,x)\text{ for }(i,x)\in \{1,2\}\times [0,x_1^*),\\
            \sum_{k=1}^2 B_kw_{jk} +v_j \text{ for } \{1,2\}\times [x_1^*,x_2^*),\\
            g(i,x) \text{ for }\{1,2\}\times [x_2^*,\infty),
            \end{cases}
        \end{align}
satisfies
\begin{enumerate}
    \item $u\geq g$,
    \item $u\in \C_b(\{1,2\}\times[0,\infty)$,
    \item $u$ is a viscosity solution to
    \begin{align}
        \min (a u(i,x)-\tmG^*u(i,0),u(i,x)-g(i,x))=0\text{ for }(i,x)\in \{(1,0),(2,0),(j,x_1^*)\}
    \end{align}
    \item $u$ is a viscosity supersolution to
    \begin{align}
        a u(i,x)-\tmG^* u(i,x)=0\text{ for }(i,x)\in \{j\}\times [x_2^*,\infty)\cup\{3-j\}\times [x_1^*,\infty)
    \end{align}
\end{enumerate}
Then, the value function $V=u$.
\end{proposition}

\begin{proof}
To show that $u$ is the viscosity solution, we divide the state space into 3 cases,
\begin{enumerate}
    \item[(i)] For $(i,x)\in \{1,2\}\times (0,x_1^*)\cup\{j\}\times(x_1^*,x_2^*)$ , the viscosity property is given by \lemref{lem:rsvpl}
    \item[(ii)] For $(i,x)\in \{(1,0),(2,0),(j,x_1^*)\}$, the viscosity property follows from condition (3)
    \item[(iii)] For $(i,x)\in \{j\}\times [x_2^*,\infty)\cup\{3-j\}\times [x_1^*,\infty)$, the viscosity property follows from condition (1) and condition (4).
\end{enumerate}
\end{proof}

Using \propref{prop:rsexpli}, we need to find $A_j,\,j=1,2,3,4, B_k,\, k=1,2,\, x_1^*$ and $x_2^*$ such that the viscosity property is satisfied at the following $5$ points; $\{(1,0),(2,0),(1,x_1^*),(1,x_2^*),(2,x_1^*)\}$ (respectively $\{(1,0),(2,0),(2,x_1^*),(2,x_2^*),(1,x_1^*)\}$), and the continuity property is satisfied at the following $3$ points $\{(1,x_1^*),(1,x_2^*),(2,x_1^*)\}$ (respectively $\{(2,x_1^*),(2,x_2^*),(1,x_1^*)\}$).

We can then derive the explicit expression of the value function as follows:
\begin{corollary}\label{corregFdi}
Let $A_j$, $B_k$,  $c_1<x_1^*<x_2^*\leq c_2$, $l\in \{1,2\}$ such that
\begin{align}
    \begin{cases}
    \sum_j A_j u_x(1,0)=0,\\
    \sum_j A_j u_{xx}(2,0)=0,\\
    \sum_j A_j u_j(l,x_1^*) = \sum_k B_k w_{lk}(1,x_1^*)+v_l(x_1^*),\\
    \sum_j A_j u_{j}(3-l,x_1^*) = g(3-l,x_1^*),\\
    \sum_k B_k w_{lk}(,x_2^*)+v_{3-l}(x_2^*)=g(l,x_2^*),\\
    \sum_j A_j (u_j)_{x}(l,x_1^*) = \sum_k B_k (w_{lk})_{x}(1,x_1^*)+(v_l)_{x}(x_1^*)
    \end{cases}
\end{align}
and $\sum_j A_j u_j-g$ has a local minimum at $(3-l,x_1^*)$ and $\sum_j B_k u_k+v_l-g$ has a local minimum at $(l,x_2^*)$. If $u\geq g$, then $u$ is the value function.
\end{corollary}
For fixed numerical values of $c_1, c_2,q_1, q_2,$ and $a$, we show in the next example that we can find the above parameters $A_j,\,j=1,2,3,4, B_k,\, k=1,2,\, x_1^*$ and $x_2^*$ and thus derive the value function.

 Assume that $a=0.1$. Figure~\ref{fig:vfc} depicts the evolution of the value functions
\begin{align}
    V(x) = \sup_{\tau}\bfE^{i,x}\Big[e^{-a\tau}\Big\{(X(t)-c_1)^+-(X(t)-c_2)^+\Big\}\Big],
\end{align}
where $c_1=1$ and $c_2=4$ for regime switching diffusion with reflected boundary and sticky boundary with the intensity matrix
\begin{align*}
\begin{bmatrix}-q_1& q_1\\ q_2&-q_2
\end{bmatrix}
=
\begin{bmatrix}
-0.1&0.1\\0.1&-0.1
\end{bmatrix},
\end{align*}
where state 1 represents the reflected boundary and state 2 represents the sticky boundary.
	\begin{figure}[H]
		\caption{Value functions against initial state}\label{fig:vfc}
		\centering
		\includegraphics[scale=0.6]{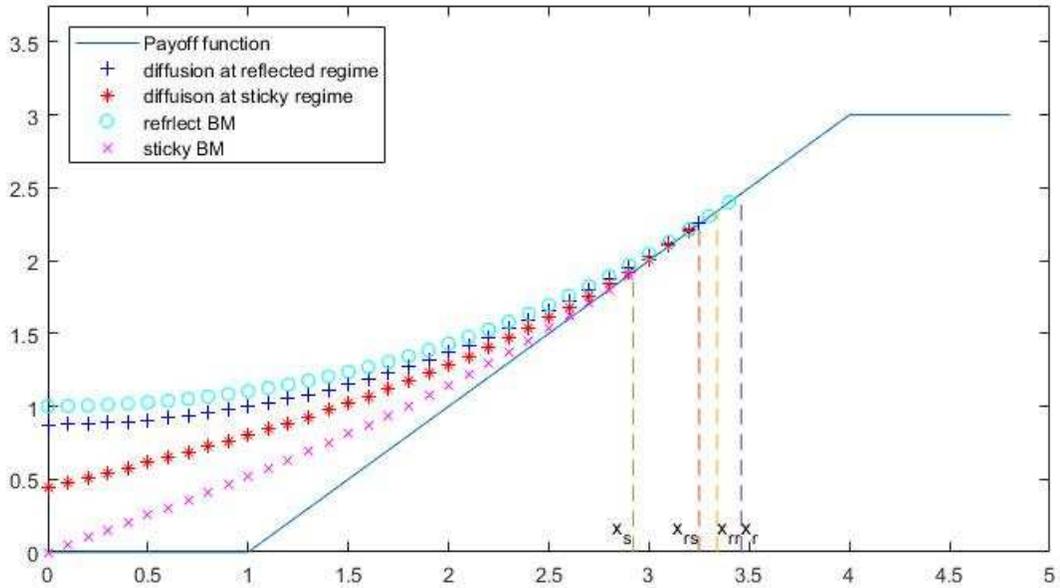}
	\end{figure}
$x_s$, $x_{rs}$, $x_{rr}$ and $x_r$ are the exercise points in the cases of sticky Brownian motion, diffusion at sticky regime, diffusion at reflected regime and reflected Brownian motion, respectively.  Figure~\ref{fig:vfc} depicts the evolution of the value function of reflected Brownian motion and sticky Brownian motion with regime switching respectively. The sticky Brownian motion has an absorbing point at $0$ and the payoff function at $0$ equals $0$. This means that the value function of the optimal stopping problem for sticky Brownian motion at $0$ is $0$ which is smaller than that of the reflected Brownian motion at $0$. Therefore, the exercise points $x_r$ for reflected Brownian motion is larger than that of the sticky Brownian motion $x_s$. The graph also shows that the value function of this regime switching process will stay between the above two value functions. This is in line with the intuition.  Additionally, the graph shows that the exercise points $x_{rs}$ and $x_{rr}$ are between $x_s$ and $x_r$.

\color{black}
\appendix

\section{Proof of \thmref{thm:osi}}
\noindent \textbf{(1)} We first prove (1) in \thmref{thm:osi}, that is, there exists a unique solution  $v_\lambda\in D(\tmL)$ to \eqnref{eqn:osipenvar} for each $\lambda>0$. Define the penalty function $v_\lambda$ as the solution to
\begin{equation}\label{eqn:osipen}
v_\lambda = \tmR_a(f+\lambda(g-v_\lambda)^+) .
\end{equation}
We start by showing that \eqref{eqn:osipen} has a unique solution in $\cz(\sE)$ in the following lemma.

\begin{lemma}\label{lem:osipen}
Suppose \assref{ass:fscompact} holds. For any $\lambda>0$, \eqref{eqn:osipen} admits a unique solution $v_\lambda\in \cz({\sE})$. Additionally, the solution to  \eqref{eqn:osipen} is equivalent to the solution to the following equation
\begin{equation}\label{eqn:resolventab1}
 v_\lambda = \tmR_{a+\lambda}(f+\lambda(g-v_\lambda)^++\lambda v_\lambda).
\end{equation}
\end{lemma}
\noindent(Equivalence here means that, the solution to one is also a solution to the other, vice versa.)

\begin{proof}
We first show that the solution to \eqref{eqn:osipen} is equivalent to the solution to \eqref{eqn:resolventab1}. Let $v_\lambda$ be the solution to \eqref{eqn:osipen} in $\cz(\sE)$. Using the resolvent identity equation \eqnref{eqn:resolventidentity}, we obtain
\begin{align}\label{eqn:resolventab1explain}
\tmR_{a+\lambda}(f+\lambda(g-v_\lambda)^+)-\tmR_{a}(f+\lambda(g-v_\lambda)^+) = -\lambda\tmR_{a+\lambda}\tmR_{a} (f+\lambda(g-v_\lambda)^+).
\end{align}
Combining \eqref{eqn:osipen} and \eqref{eqn:resolventab1explain}, we have
$$\tmR_{a+\lambda}(f+\lambda(g-v_\lambda)^+)-v_{\lambda} = -\lambda\tmR_{a+\lambda}v_\lambda.$$
Therefore, $v_\lambda$ is also a solution to \eqnref{eqn:resolventab1}.\\
Now, let $v_\lambda$ be a solution to \eqnref{eqn:resolventab1}. Using once more \eqnref{eqn:resolventidentity}, we have
\begin{align}\label{eqn:resolventab2explain}
&\tmR_{a+\lambda}(f+\lambda(g-v_\lambda)^++\lambda v_\lambda)-\tmR_{a}(f+\lambda(g-v_\lambda)^++\lambda v_\lambda)\nonumber\\
=&-\lambda\tmR_{a}\tmR_{a+\lambda}(f+\lambda(g-v_\lambda)^++\lambda v_\lambda).
\end{align}
Combining \eqref{eqn:resolventab1} and \eqref{eqn:resolventab2explain} yields
$$v_\lambda-\tmR_{a}(f+\lambda(g-v_\lambda)^++\lambda v_\lambda) = -\lambda\tmR_a v_\lambda.$$
Hence, $v_\lambda$ is also a solution to \eqref{eqn:osipen}.

In order to complete the proof, it is enough to show that \eqref{eqn:resolventab1} has the unique solution. Define a new operator $\mathcal{Z}$ as  follows:
$$\mathcal{Z}w := \tmR_{a+\lambda}(f+\lambda(g-w)^++\lambda w).$$
We have that $f,g \in \cz(\sE)$ and the resolvent operator maps from $\cz(\sE)$ to $\cz(\sE)$. Let $w\in \cz(\sE)$, then $\mathcal{Z}w$ is also in $\cz(\sE)$. Furthermore, let $w_1,w_2\in\cz{(\sE)}$. Using the linearity of the resolvent and the fact that $(g-w_i)^++w_i = \max(g,w_i)$ for $i=1,2$, we have
\begin{eqnarray*}
\ninf{\mathcal{Z}w_1-\mathcal{Z}w_2}&=&\ninf{\tmR_{a+\lambda}(f+\lambda(g-w_1)^++\lambda w_1)-\tmR_{a+\lambda}(f+\lambda(g-w_2)^++\lambda w_2)}\\
&=& \ninf{\lambda\tmR_{a+\lambda}(\max{(g,w_1)}-\max{(g,w_2)})}\\
&\leq& \frac{\lambda}{a+\lambda}\ninf{\max{(g,w_1)}-\max{(g,w_2)})}\\
&\leq& \frac{\lambda}{a+\lambda}\ninf{w_1-w_2},
\end{eqnarray*}
where the inequality comes from \eqref{eqn:resolventproperties233}. Hence, $\mathcal{Z}$ is a contraction mapping from $\cz(\sE)$ to $\cz(\sE)$. By Banach fixed point theorem, the equation $w = \mathcal{Z}w$ (which is the same as \eqnref{eqn:resolventab1}) has a unique solution $w\in \cz(\sE)$, that we denote by $v_\lambda$.
\end{proof}

 Recall that the operator $(\lambda-\tmL)$ is a bijection of $D(\tmL) $ to $\cz(\sE)$ and its inverse is the resolvent $\tmR_\lambda$ (see \corref{cor:hy}). The solution to \eqref{eqn:osipen} is equivalent to the solution to \eqref{eqn:osipenvar}. It remains to show that $v_\lambda \in D(\tmL)$. We have shown that \eqref{eqn:osipen} has the unique solution $v_\lambda$ in $\cz(\sE)$ such that $f+\lambda(g-w)^+\in \cz(\sE)$. Therefore, \textbf{(1)} in \thmref{thm:osi} is proved.

\noindent \textbf{(2)} Let $v_\lambda$ be the unique solution in $D(\tmL)$ to \eqref{eqn:osipenvar} for $\lambda>0$. We prove that the sequence of penalty functions $\{v_\lambda\}_{\lambda>0}$ converges uniformly to the value function $V$ in $\cz(\sE)$ as $\lambda\rightarrow \infty$. We need the following two preliminary lemmas.

\begin{lemma}\label{lem:vnu}
Suppose \assref{ass:fscompact} holds. Let $\{g_n\}_{n\mN}$ be a sequence  in $\cz(\sE)$ such that
\begin{align}
\ninf{g-g_n}\leq \frac{1}{n}.
\end{align} Define a sequence of the corresponding value functions $\{V_n\}_{n\in\mN}$ by
\begin{equation}\label{eqn:vnu}
V_n(x) := \sup_{\tau\in\mathcal{T}} J^{(n)}_x(\tau) \text{ for }x\in \sE\text{ and } n\in\mN,
\end{equation}
where $J^{(n)}_x(\tau) = \bfE^x\left[\int_{0}^{\tau}{e^{-as}f(X(s))ds}+e^{-a\tau}g_n(X(\tau))\right].$ Then, $V_n$ converges to $V$ defined by \eqref{eqn:valuefunction}  uniformly as $n\rightarrow \infty$.
\end{lemma}
\begin{proof}
Let $x\in \sE$, $n\in\mN$ and $\varepsilon>0$. Define a $\varepsilon$-optimal stopping time $\tau_\varepsilon^*$ such that
\begin{equation}\label{eqn:tauvarepsilon}
V(x)-\varepsilon\leq J_x(\tau_\varepsilon^*).
\end{equation}
Therefore, we have
\begin{align*}
V(x) &\leq  \bfE^{x}\bigg[\int_{0}^{\tau_\varepsilon^*}e^{-as}f(X(s))ds+e^{-a\tau_\varepsilon^*}g(X(\tau_\varepsilon^*))\bigg]+\varepsilon\\
&\leq \bfE^{x}\bigg[\int_{0}^{\tau_\varepsilon^*}e^{-as}f(X(s))ds+e^{-a\tau_\varepsilon^*}(g_n(X(\tau_\varepsilon^*))+\ninf{g-g_n})\bigg]+\varepsilon\\
&\leq J_x^{(n)}(\tau_\varepsilon^*)+\frac{1}{n}+\varepsilon\leq V_n(x)+\frac{1}{n}+\varepsilon.
\end{align*}
Since $\varepsilon$ is an arbitrary positive constant, $V(x)-V_n(x)\leq \frac{1}{n}$.
On the other hand, we can find a stopping time $\tau_\varepsilon^{*(n)}$ for $V_n$ such that $V_n(x)-\varepsilon\leq J^{(n)}_x(\tau_\varepsilon^{*(n)}).$ One also obtains $V_n(x)-V(x)\leq \frac{1}{n}$ similarly.
Therefore, we have
\begin{align}\label{eqn:vnulambdatovlambda1}
\ninf{V-V_n}\leq \frac{1}{n}.
\end{align}
Then, the proof is completed.
\end{proof}

\begin{lemma} \label{lem:osipencon}
Suppose \assref{ass:fscompact} holds. Let $v_\lambda$ be the solution to \eqnref{eqn:osipen} for each $\lambda>0$. Then, $v_\lambda$ satisfies
\begin{equation}\label{eqn:osipenieq}
v_\lambda(x) = \sup_{\tau}\bfE^x\bigg[\int_{0}^{\tau}{e^{-as}f(X(s))ds}+e^{-a\tau}(g-(g-v_{\lambda})^+)(X(\tau))\bigg].
\end{equation}
Additionally, $V\geq v_\lambda$.
\end{lemma}
\begin{proof}
Let  $x\in\sE$, $\lambda>0$ and $\tau$ be a $\mF_t$-stopping time. We know from \textbf{(1)} in \thmref{thm:osi} that $v_\lambda\in D(\tmL)$. Using Dynkin's formula and optional stopping theorem,
\begin{align}\label{eqn:proofosipencon1}
v_\lambda(x) &= \bfE^x\Big[\int_{0}^{\tau}{e^{-as}(f(X(s))+\lambda(g-v_\lambda)^+(X(s)))ds}+e^{-a\tau}v_\lambda(X(\tau))\Big].\\
&\geq \bfE^x\bigg[\int_{0}^{\tau}{e^{-as}f(X(s))ds}+e^{-a\tau}(g-(g-v_{\lambda})^+)(X(\tau))\bigg].\nonumber
\end{align}
Taking the supremum on both sides, we obtain
\begin{equation}\label{eqn:vlambdaproof}
v_\lambda(x)\geq \sup_{\tau}\bfE^x\Big[\int_{0}^{\tau}{e^{-as}f(X(s))ds}+e^{-a\tau}(g-(g-v_{\lambda})^+)(X(\tau))\Big].
\end{equation}
In order to prove the equality,  define the stopping time $\sigma^*$ by $ \sigma^*:= \inf\{s\geq 0 ;\;v_\lambda(X(s))\leq g(X(s))\}$. Since $\{X(t)\}_{t\geq 0}$ is right continuous and $v_\lambda$ and $g$ are continuous, we have $v_\lambda(X(\sigma^*))\leq g(X(\sigma^*))$. Using the preceding and \eqnref{eqn:proofosipencon1}, we have
\begin{align*}
v_\lambda(x) &= \bfE^x\Big[\int_{0}^{\sigma^*}{e^{-as}(f(X(s))+\lambda(g-v_\lambda)^+(X(s)))ds}+e^{-a\sigma^*}v_\lambda(X(\sigma^*))\Big]\\
&= \bfE^x\Big[\int_{0}^{\sigma^*}{e^{-as}f(X(s)ds}+e^{-a\sigma^*}(g-(g-v_{\lambda})^+)(X(\sigma^*))\Big].
\end{align*}
Hence, \eqref{eqn:osipenieq} is proved. Furthermore, since $g-(g-v_\lambda)^+\leq g$, \eqref{eqn:osipenieq} implies $V\geq v_\lambda$ for all $\lambda>0$. The proof is completed.
\end{proof}

\begin{lemma}\label{lem:gdvltov}
Suppose \assref{ass:fscompact} holds. Assume in addition that $g\in D(\tmL)$. Let $v_\lambda$ be the solution to \eqref{eqn:osipenvar} for $\lambda>0$. $v_\lambda$ converges to $V$ uniformly as $\lambda\to\infty$ and hence $V\in \cz(\sE)$.
\end{lemma}
\begin{proof}
Let $x\in\sE$ and $\lambda>0$. Using \eqnref{eqn:osipenieq}, we get
\begin{align*}
v_\lambda(x) &= \sup_{\tau}{\bfE^x\Big[ \int_{0}^{\tau}{e^{-as}f(X(s))ds}+e^{-a\tau}(g-(g-v_{\lambda})^+)(X(\tau)) \Big]}\nonumber\\
&\geq \sup_{\tau}{\bfE^x\Big[ \int_{0}^{\tau}{e^{-as}f(X(s))ds}+e^{-a\tau}g(X(\tau)) \Big]} -\sup_{\tau}{\bfE^x\Big[ e^{-a\tau}(g-v_{\lambda})^+(X(\tau)) \Big]}\nonumber\\
&\geq V(x) - \sup_{\tau}{\bfE^x\left[ (g-v_{\lambda})^+(X(\tau)) \right]} \nonumber\\
&\geq V(x) - \ninf{(g-v_\lambda)^+}.
\end{align*}
Additionally, we have from \lemref{lem:osipencon} that $V\geq v_\lambda$ for all $\lambda>0$. Then,
\begin{equation}\label{eqn:proofinequalityvvlambda}
\ninf{V-v_\lambda}\leq \ninf{(g-v_\lambda)^+}.
\end{equation}
Furthermore, since $v_\lambda\in D(\tmL)$ by \textbf{(1)} in \thmref{thm:osi}, $g-v_\lambda\in D(\tmL)$ and thus $$g-v_\lambda = \tmR_a((a-\tmL)(g-v_\lambda)).$$
Hence, using  \eqnref{eqn:osipenvar} and similar argument as in \eqref{eqn:resolventab1explain}, we get
\begin{align}\label{eqn:corprf}
g-v_\lambda&=\tmR_a((a-\tmL )g-(a-\tmL)v_\lambda)\nonumber\\
&=\tmR_a((a-\tmL )g-f-\lambda(g-v_\lambda)^+)\nonumber\\
&= \tmR_{a+\lambda}((a-\tmL )g-f-\lambda(g-v_\lambda)^++\lambda(g-v_\lambda))\nonumber\\
&\leq  \tmR_{a+\lambda}((a-\tmL )g-f-\lambda(g-v_\lambda)^++\lambda(g-v_\lambda)^+)\\
&\leq \frac{\ninf{(a-\tmL )g-f}}{a+\lambda}.\nonumber
\end{align}
Thus, it follows from \eqnref{eqn:proofinequalityvvlambda} that
\begin{align}\label{eqn:vnulambdatovlambda2}
\ninf{V-v_\lambda}\leq\ninf{(g-v_\lambda)^+}\leq  \frac{\ninf{(a-\tmL )g-f}}{a+\lambda}\leq \frac{M}{a+\lambda},
\end{align}
where $M>0$ is a constant. Hence, $v_\lambda$ converges to $V$ uniformly as $\lambda\rightarrow\infty$. The proof of \textbf{(1)} in \thmref{thm:osi} is completed.
\end{proof}
It remains to show that the conclusion of \lemref{lem:gdvltov} is also true for any $g\in \cz(\sE)$. Let $g\in \cz(\sE)$. $D(\tmL)$ is dense in $\cz(\sE)$ (see  \thmref{thm:hyr}). Thus there exists a sequence $\{g_n\}_{n\in\mN}$ in $D(\tmL)$ uniformly converging to $g$ as $n\rightarrow \infty$ such that $\ninf{g_n-g}\leq 1/n$ for any $n\in \mN$. Using \lemref{lem:gdvltov}, the sequence of the value functions $\{V_n\}_{n\in\mN}$ defined by \eqnref{eqn:vnu} is in $\cz(\sE)$. Using \lemref{lem:vnu}, $\{V_n\}_{n\in \mN}$ converges to $V$ uniformly as $n\rightarrow\infty$. Therefore, $V\in\cz(\sE)$.

Moreover, let $v_{n,\lambda}$ be the solution to \eqref{eqn:osipenvar} after replacing $g$ by $g_n$ for each $n\in \mN$. We prove that $v_{n,\lambda}$ converges to $v_\lambda$ uniformly as $n\rightarrow \infty$.  Using \eqref{eqn:resolventab1},
\begin{align*}
v_{n,\lambda}-v_\lambda &= \tmR_{a+\lambda}(f+\lambda(g_n-v_{n,\lambda})^++\lambda v_{n,\lambda})-\tmR_{a+\lambda}(f+\lambda(g-v_\lambda)^++\lambda v_\lambda)\\
& =  \lambda \tmR_{a+\lambda}(\max(g_n,v_{n,\lambda})-\max(g,v_\lambda))\\
&\leq \lambda \tmR_{a+\lambda}(\max(g_n-g,v_{n,\lambda}-v_\lambda))\\
&\leq \frac{\lambda}{a+\lambda} \ninf{\max(g_n-g,v_{n,\lambda}-v_\lambda)}\\
&\leq \frac{\lambda}{a+\lambda}\max(\ninf{g-g_n},\ninf{v_{n,\lambda}-v_\lambda}).
\end{align*}
Similarly, we also have
\begin{align*}
v_\lambda-v_{n,\lambda} &= \tmR_{a+\lambda}(f+\lambda(g-v_\lambda)^++\lambda v_\lambda)-\tmR_{a+\lambda}(f+\lambda(g_n-v_{n,\lambda})^++\lambda v_{n,\lambda})\\
&\leq \frac{\lambda}{a+\lambda}\max(\ninf{g-g_n},\ninf{v_{n,\lambda}-v_\lambda}).
\end{align*}
Thus
\begin{align*}
\ninf{v_\lambda-v_{n,\lambda}}\leq  \frac{\lambda}{a+\lambda}\max(\ninf{g-g_n},\ninf{v_{n,\lambda}-v_\lambda}).
\end{align*}
Since $\frac{\lambda}{a+\lambda}<1$, it follows that
\begin{eqnarray}\label{eqn:vnulambdatovlambda}
\ninf{v_{n,\lambda}-v_\lambda} \leq \ninf{g_n-g}\leq \frac{1}{n}.
\end{eqnarray}
Then, $v_{n,\lambda}$ converges to $v_\lambda$ uniformly as $n\rightarrow \infty$.

 Now, let $\varepsilon>0$ and choose an integer $n_0>\frac{4}{\varepsilon}$. Hence, combing \eqref{eqn:vnulambdatovlambda1}, \eqref{eqn:vnulambdatovlambda2} and  \eqref{eqn:vnulambdatovlambda} yields
\begin{align}
\ninf{V-v_\lambda}&\leq \ninf{V-V_{n_0}}+\ninf{v_{n_0,\lambda}-v_\lambda}+\ninf{V-v_{n_0,\lambda}}\nonumber\\
&\leq \frac{1}{n_0}+\frac{1}{n_0}+\frac{M_{n_0}}{a+\lambda}\leq \frac{\varepsilon}{2}+\frac{M_{n_0}}{a+\lambda},
\end{align}
where $M_{n_0}=\ninf{(a-\tmL)g_{n_0}-f}$. Therefore, $\ninf{V-v_\lambda}\leq \varepsilon$ for any $\lambda>\frac{2M_{n_0}}{\varepsilon}$. Thus, the proof is completed.

\section{Proof of \lemref{lem:coreoperator}}
	Let $\{\mP_t\}_{t\geq 0}$ be the Feller semigroup defined by \eqnref{eqn:extfs}, that is for any $w\in \C(\sO)$,
	\begin{equation}\label{eqn:extfs1}
	\mP_t w:= w(\partial) +\reallywidetilde{\tmP_t ((w-w(\partial))|_\sE)}.
	\end{equation}
	By \defref{def:infgendef}, its infinitesimal generator  $(\mL,D(\mL))$ can be defined by:
	\begin{align}
	\mL w:=& \lim_{t\rightarrow 0^+}\frac{\mP_t w- w}{t} \mbox{  for each }w \in D(\mL),\label{eqn:mLdefine00}\\
	D(\mL):=& \{w\in \C(\sO); \lim_{t\rightarrow 0^+}\frac{\mP_t w- w}{t} \mbox{ exists in }\C(\sO)\} .\label{eqn:mLdefine10}
	\end{align}
	Let $D_0$ be a domain defined by
	\begin{equation}\label{eqn:defd0}
	D_0: = \{w\in \C(\sO); (w-w(\partial))|_\sE\in D(\tmL)\}.
	\end{equation}
	We show that  $D(\mL)=D_0$ by double inclusion. We first prove that $D_0 \subseteq D(\mL)$. Let $w\in D_0\subseteq \C(\sO)$. We have  by \eqnref{eqn:extfs1} restricted on $\sE$ that
	\begin{eqnarray}\label{eqn:prooflimitpp}
	\lim_{t\rightarrow 0^+} \frac{(\mP_t w)|_\sE - w|_\sE}{t} = \lim_{t\rightarrow 0^+} \frac{\tmP_t( (w-w(\partial))|_E)+w(\partial) - w|_\sE}{t}.
	\end{eqnarray}
	Since $w\in D_0$, we have $(w - w(\partial))|_{\sE}\in D(\tmL)$, then the limit on the right hand side of \eqnref{eqn:prooflimitpp} exists in $\cz(\sE)$ and we can write
	\begin{equation}\label{eqn:prooflimitpp1}
	\lim_{t\rightarrow 0^+} \frac{(\mP_t w)|_\sE - w|_\sE}{t}= \tmL ( (w-w(\partial))|_E)\in \cz(\sE).
	\end{equation}
	In addition, using \eqnref{eqn:extfs1} and the fact that $(w-w(\partial))|_\sE\in \cz(\sE)$, $\mP_t w(\partial) = w(\partial)$ for all $t\geq 0$. Hence, we know that
	\begin{equation}\label{eqn:prooflimitpp11}
	\lim_{t\rightarrow 0^+} \frac{\mP_t w(\partial) - w(\partial)}{t}= 0.
	\end{equation}
	Putting \eqref{eqn:prooflimitpp1} and \eqref{eqn:prooflimitpp11} together yields for any $w\in D_0$,
	$\underset{t\rightarrow 0^+}{\lim}\frac{\mP_t w- w}{t}$ exists in $\C(\sO)$ and by the definition of the extension, we get
	\begin{equation}\label{eqn:prooflimitpp2}
	\lim_{t\rightarrow 0^+}\frac{\mP_t w- w}{t} = \reallywidetilde{\tmL  ((w-w(\partial))|_\sE)} \mbox{ exists in }\C(\sO)\mbox{  for any }w\in D_0.
	\end{equation}
	Thus, $D_0 \subseteq  D(\mL)$.
	
	Let us now prove that $D(\mL)\subseteq D_0$. Choose $w\in D(\mL)$. Then, since for such $w$, the limit of \eqnref{eqn:mLdefine00} exists in $\C(\sO)$, it follows that the limit on the right hand side of \eqnref{eqn:prooflimitpp} also exists. In addition, using \eqref{eqn:mLdefine10} and \eqref{eqn:prooflimitpp11} respectively, we have $\mL w\in \C(\sO)$ and $\mL w(\partial) = 0$ and thus the limit
	$$\lim_{t\rightarrow 0^+}\frac{\tmP_t ((w-w(\partial))|_E)+ (w-w(\partial))|_E}{t} \text{ exists in } \cz(\sE).$$
	Therefore, due to the fact that $(w-w(\partial))|_\sE\in\cz(\sE)$, we have $(w-w(\partial))|_\sE\in D(\tmL)$ which means $w\in D_0$.
	We can conclude that $D(\mL) = D_0$ and $\mL$ is given by \eqnref{eqn:prooflimitpp2}, that is \eqref{eqn:mLdefine} and \eqref{eqn:mLdefine1} are proved.
	
	Then, it is reasonable to define the restriction of $(\mL,D(\mL))$ on $D(\mG)$ by \eqref{eqn:tgdefine0}.
	
	Let us now show that $(\mG,D(\mG))$ is the core of $(\mL,D(\mL))$. Suppose that there is a sequence $\{w_n\}_{n\in\mN}$ in $D(\mG)$ satisfying $w_n\rightarrow u$ and $\mG w_n \rightarrow v$ uniformly in $\C(\sO)$. It is enough to prove that $u\in  D(\mL)$ and $v = \mL u$. Using \eqnref{eqn:tgdefine0}, the sequence $\{w_n^*\}_{n\in\mN}$ defined by
	 $$w_n^*:=(w_n-w_n(\partial))|_\sE \text{ for } n\in \mN$$
	  belongs to $D(\tmG)$ and satisfies $w^*_n\rightarrow (u-u(\partial))|_\sE$ and $\tmG w^*_n \rightarrow v|_\sE$ uniformly in $\cz(\sE)$. In addition, since $(\tmG,D(\tmG))$ is the core of $(\tmL,D(\tmL))$, it follows that $(u-u(\partial))|_\sE\in D(\tmL)$ and $v|_\sE = \tmL((u-u(\partial))|_\sE)$. Therefore, using \eqnref{eqn:mLdefine00}, $u\in  D(\mL)$ and $v = \mL u$. The proof is completed.

\bibliographystyle{apa}
\bibliography{Manuscript}

\end{document}